\documentclass[10t, reqno]{amsart}
\usepackage[margin=1.1in]{geometry}

\usepackage{enumerate, amsmath, amsfonts, amssymb, xy,  mathrsfs, graphicx, paralist, fancyvrb,mathtools}
\usepackage[usenames, dvipsnames]{xcolor}
\usepackage[bookmarks, colorlinks=true, linkcolor=blue, citecolor=blue, urlcolor=blue]{hyperref}

\usepackage{graphicx,url,etoolbox}

\input xy
\xyoption{all}

\usepackage{tikz,tikz-cd}	
\usetikzlibrary{positioning, matrix, shapes}  

\newcommand{\cB}{\mathcal{B}}

\newcommand{\bC}{\mathbf{C}}

\newcommand{\fC}{\mathfrak{C}}

\newcommand{\bN}{\mathbf{N}}

\newcommand{\cS}{\mathcal{S}}
\newcommand{\fS}{\mathfrak{S}}

\newcommand{\bV}{\mathbf{V}}

\newcommand{\bW}{\mathbf{W}}

\newcommand{\bZ}{\mathbf{Z}}

\newcommand{\be}{\mathbf{e}}

\newcommand{\bk}{\mathbf{k}}

\renewcommand{\phi}{\varphi}
\renewcommand{\emptyset}{\varnothing}

\newcommand{\injects}{\hookrightarrow}
\newcommand{\surjects}{\twoheadrightarrow}

\newcommand{\ol}[1]{\overline{#1}}

\DeclareMathOperator{\End}{End}

\newcommand{\GL}{\mathbf{GL}}

\newcommand{\lO}{\mathbf{O}}

\newcommand{\Spin}{\mathbf{Spin}}
\newcommand{\Pin}{\mathbf{Pin}}

\newcommand{\fgl}{\mathfrak{gl}}
\newcommand{\fso}{\mathfrak{so}}

\numberwithin{equation}{subsection}
\newtheorem{theorem}[equation]{Theorem}

\newtheorem{lemma}[equation]{Lemma}
\newtheorem{corollary}[equation]{Corollary}

\theoremstyle{definition}
\newtheorem{rmk}[equation]{Remark}
\newenvironment{remark}[1][]{\begin{rmk}[#1] \pushQED{\qed}}{\popQED \end{rmk}}
\newtheorem{eg}[equation]{Example}
\newenvironment{example}[1][]{\begin{eg}[#1] \pushQED{\qed}}{\popQED \end{eg}}
\newtheorem{defn}[equation]{Definition}
\newenvironment{definition}[1][]{\begin{defn}[#1]\pushQED{\qed}}{\popQED \end{defn}}

\newcommand{\SB}{{\rm \bf SB}}

\begin{document}

\title{The spin-Brauer diagram algebra}


\author{Robert P. Laudone
}


\address{Department of Mathematics, University of Wisconsin, Madison}
\email{laudone@wisc.edu \newline \indent {\em URL:} \url{https://www.math.wisc.edu/~laudone/}}
\date{\today}

\maketitle

\begin{abstract}
We investigate the spin-Brauer diagram algebra, denoted $\SB_n(\delta)$, that arises from studying an analogous form of Schur-Weyl duality for the action of the pin group on $\bV^{\otimes n} \otimes \Delta$. Here $\bV$ is the standard $N$-dimensional complex representation of $\Pin(N)$ and $\Delta$ is the spin representation. When $\delta = N$ is a positive integer, we define a surjective map $\SB_n(N) \surjects \End_{\Pin(N)}(\bV^{\otimes n} \otimes \Delta)$ and show it is an isomorphism for $N \geq 2n$.  We show $\SB_n(\delta)$ is a cellular algebra and use cellularity to characterize its irreducible representations.
\end{abstract}

\section{Introduction} \label{introduction}
Schur-Weyl duality is a seminal result in representation theory. It states that the actions of $\Sigma_n$ and $\GL(N)$ on $\bV^{\otimes n}$ generate each others' commutators.  Here $\Sigma_n$ is the symmetric group on $n$ letters and $\bV$ is the standard representation. 

In \cite{brauer1937algebras}, Brauer pursued an analogous result to Schur-Weyl duality, replacing the general linear group with the orthogonal group.  Using invariant theory, he proved that the Brauer diagram algebra, denoted $\mathbf{B}_{n}(\delta)$, surjects onto $\End_{\lO(N)}(\bV^{\otimes n})$. Here $\delta$ is the dimension of $\bV$.  These diagram algebras, however, are well defined for any parameter $\delta$ and over any commutative ring $R$.  Brauer proved this in his paper by giving a purely combinatorial description of multiplication in $\mathbf{B}_{n}(\delta)$. He went on to show $\mathbf{B}_n(\delta)$ possessed certain properties, but questions about its semi-simplicity and irreducible representations were not well understood until recently.

In \cite{koike2005spin}, Koike pursued another analogue of Schur-Weyl duality, replacing $\lO(N)$ with its double cover $\Pin(N)$. He constructed a diagram algebra, and proved that it surjects onto $\End_{\Pin(N)}(\bV^{\otimes n} \otimes \Delta)$ and is a bijection for $N$ sufficiently large with $n$ fixed.  Here $\bV$ is the standard representation of the orthogonal group and $\Delta$ is the spin representation of $\Pin(N)$.  As in the above example, studying a centralizer algebra gave rise to a diagram algebra. However, in this case Koike showed some diagram algebra surjects onto the centralizer and is a bijection for sufficiently large $N$, but never gave an explicit combinatorial structure to this diagram algebra and never proved the map was a homomorphism. Without this structure, this version of Schur-Weyl duality was largely incomplete.

Accordingly, the purpose of this paper is to complete this form of Schur-Weyl duality by defining and studying an associative algebra $\SB_n(\delta)$ equipped with a purely combinatorial multiplication structure which we prove is equivalent to the diagram algebra Koike vaguely defines. The existence of additional equivariant maps to and from $\Delta$ complicates the composition structure of $\End_{\Pin(N)}(\bV^{\otimes n} \otimes \Delta)$.  Consequently, unlike other diagram algebras mentioned in \S \ref{relation}, defining a multiplication structure on $\SB_n(\delta)$ is not a straight-forward extension of multiplication in ${\bf B}_n(\delta)$ or any other diagram algebra. For example, one may can note that the basis of diagrams that arise for $\SB_n(\delta)$ is naturally contained in the diagrammatic basis for the partition algebra. However, the multiplication structure we define does not make $\SB_n(\delta)$ a subalgebra.

In section \ref{Diagram Algebra}, we provide this purely combinatorial description of $\SB_n(\delta)$. In particular, we give a multiplication structure on its basis elements and prove this structure is associative. This allows us to define the spin-Brauer diagram algebra over any commutative ring, with $n \in \bZ_{\geq 0}$ and for any parameter $\delta$. We then prove our main theorem:

\begin{theorem} \label{TheoremA}
For $n,N \in \bZ^+$, $\SB_n(N)$ surjects onto the centralizer algebra $\End_{\Pin(N)}(\bV^{\otimes n} \otimes \Delta)$ and for $N \geq 2n$ the map in (\ref{mapconvert}) is an isomorphism.
\end{theorem}



This shows our notion of combinatorial multiplication is correct and completes this form of Schur-Weyl duality for the Pin group. In the section \ref{Cellularity}, we use the work developed in \cite{graham1996cellular} in combination with \cite{xi1999partition} to provide a basis free proof that $\SB_n(\delta)$ is cellular. We then use Graham and Lehrer's work to describe an indexing set for the irreducible representations of $\SB_n(\delta)$.

%

\subsection{Outline of Argument}
The proof of Theorem \ref{TheoremA} breaks into the following steps:
\begin{enumerate}
\item In section \ref{Diagram Algebra}, we provide a purely combinatorial description of $\SB_n(\delta)$ and prove that with this multiplication structure, $\SB_n(\delta)$ is an associative algebra.
\item In section \ref{ProjInjMaps}, we explicitly construct the $\Pin(N)$-equivariant projection $V \otimes \Delta \twoheadrightarrow \Delta$ (also discussed in \cite{sam2016infinite}) and injection $\Delta \hookrightarrow V \otimes \Delta$. We then prove equivariance.
\item In section \ref{Multiplication Agrees}, we use this explicit construction to prove that the combinatorial multiplication structure we define in Section \ref{Diagram Algebra} agrees with the composition of maps the diagrams represent, as described in Section \ref{ProjInjMaps}.
\item We then prove the equivalence of our diagram algebra with the diagram algebra Koike mentions in his paper and use this to deduce that $\SB_n(N)$ surjects onto $\End_{\Pin(N)}(\bV^{\otimes n} \otimes \Delta)$ and is an isomorphism for $N \geq 2n$.
\end{enumerate}

\subsection{Relation to Previous Work} \label{relation}
Following Brauer's work, many variations of Schur-Weyl duality for matrix subgroups of $\GL(V)$ and their corresponding centralizer algebras were investigated.
\begin{itemize}[$\bullet$]
\item Koike \cite{KoiWB} and Turaev \cite{Tur} independently discovered the walled-Brauer diagram algebra, ${\rm Br_{r,t}(N)}$ as the centralizer of $\GL(V)$ on $V^{\otimes r} \otimes  (V^\ast)^{\otimes t}$. It has since been highly studied.  For example, in \cite{BCHLLS} the authors decompose $V^{\otimes r} \otimes  (V^\ast)^{\otimes t}$ into irreducible $\GL(V)$-modules. Then, in \cite{CVDM2008} Cox, De Visscher, Doty and Martin discuss its blocks and semi-simplicity.
\item Martin \cite{Mar1,Mar2} and Jones \cite{Jon} independently discovered the partition diagram algebra.  It arose within the context of statistical mechanics as the centralizer of the action of $\fS_n$ on $\bV^{\otimes k}$, the $k$-fold tensor product of the $n$-dimensional permutation representation representation $\bV$. In \cite{HR}, Halverson and Ram provided an explicit presentation by generators and relations and showed the existence of Murphy elements.
\item Diagram algebras often emerge as the space of morphisms in Deligne categories. For an introduction to Deligne categories, we refer the reader to \cite{Et1,Et2}. In particular, the Brauer diagram algebras describe the morphisms in the category ${\rm Rep}(\GL_t)$ and the partition diagram algebras describe the morphisms in the category ${\rm Rep}(\Sigma_t)$. We expect a corresponding theory for the spin-Brauer algebra interpolating the categories of representations of the group $\Pin_n$ as $n \in \bN$ varies.
\item In \cite{graham1996cellular}, Graham and Lehrer defined the notion of a cellular algebra and proved that $\mathbf{B}_n(\delta)$ is cellular.  Many other diagram algebras were proved to be cellular. For example, both the partition algebra \cite{xi1999partition} and walled Brauer diagram algebra \cite[Theorem 2.7]{CVDM2008} are cellular.
\end{itemize}

\subsection{Conventions and Background}
Unless otherwise stated we will be working over a field $\bk$ of characteristic $0$. We always use $\bV$ to denote the complex $N$-dimensional standard representation of the orthogonal group $\lO(N)$ or equivalently the standard representation of $\Pin(N)$. $\Delta$ will denote the spin representation of $\Pin(N)$, which we define more explicitly below. 

We assume a basic knowledge of the Clifford Algebra $\fC(Q)$ and the pin group $\Pin(N,Q)$, where $Q$ is a bilinear form. We will define a bilinear form in Section \ref{Spinor Rep} and use this same bilinear form throughout the paper.  As a result, we suppress the bilinear form in $\Pin(N,Q)$, just writing $\Pin(N)$. We will use the fact that $\Pin(N)$ is the simply-connected double cover of the orthogonal Lie group $\lO(N)$ with associated Lie algebra $\fso(N)$. Additionally, we recall $\Pin(N)$ is not connected.  Indeed, it has two connected components given by $\Pin(N) \cap \fC^{\text{even}}(Q)$ and $\Pin(N) \cap \fC^{\text{odd}}(Q)$. 

Furthermore, the subgroup $\Spin(N)  \simeq \Pin(N) \cap \fC^{\text{even}}(Q) \subset \Pin(N)$ is a connected and simply-connected Lie group with Lie algebra $\fso(N)$.   Being a connected and simply connected Lie group, studying $\Spin(N)$-equivariant maps is equivalent to studying $\fso(N)$-equivariant maps. If we prove $\fso(N)$-equivariance, we can then deduce $\Pin(N)$-equivariance by checking equivariance for one element in $\Pin(N) \cap \fC^{\text{odd}}(Q)$. Indeed, this element will generate the odd degree subspace of $\Pin(N) \cap \fC^{\text{odd}}(Q)$ as a $\Spin(N)$-algebra.

 
We take this perspective because proving $\fso(N)$-equivariance is easier and more illuminating than working with the spin or pin groups. We identify $\fso(N)$ with $\bigwedge^2 \bV$,  where $\bV$ is the $N$-dimensional standard representation. For further background the reader might consult \cite[$\S$ 20]{fulton2013representation}.

\subsection{Acknowledgements}
I thank Steven V Sam for his constant guidance and helpful conversations. I also thank two anonymous referees for many helpful comments that greatly improved the exposition of this paper. This work was supported by NSF grant DMS-1502553.

\section{The Spinor Representation} \label{Spinor Rep}

\subsection{Basic Definitions} We begin by defining $\fso(N)$ for both even and odd-dimensional standard representations $\bV$ similar to \cite{sam2016infinite}. We will make this definition and then use it to explicitly describe the action of $\fso(N)$ on the standard representation. Let $\bW = \bC^m$ and $\bW^\ast = (\bC^m)^\ast$ its dual and put
\[
\ol{\bV} = \bW \oplus \bW^\ast, \qquad \bV = \ol{\bV} \oplus \bC.
\]
Let $\be$ be a basis vector for the one dimensional space $\bC$ of $\bV$. We define an orthogonal form on $\bV$ so that $\bW$ and $\bW^\ast$ are both $m$-dimensional isotropic subspaces of $\bV$ and $\bC$ is a one-dimensional space perpendicular to both of them.  Define the orthogonal form $\omega'$ on $\ol{\bV}$ by
\[
\omega'((v,f),(v',f')) = f'(v) + f(v').
\]
Extend this to an orthogonal form $\omega$ on $\bV$ by setting $\omega(\be,\be) = 1$ and $\omega(\be,v) = 0$ for all $v \in \ol{\bV}$. For even $N$, when we discuss the standard representation we mean the vector space $\ol{\bV}$ with orthogonal form $\omega'$, for odd $N$ we will mean the vector space $\bV$ with orthogonal form $\omega$. 

Now we are ready to describe $\fso(N)$. Recall there is an isomorphism of $\fso(N)$ with the second exterior power of the standard representation, as seen for example in \cite[\S20.1]{fulton2013representation}. Using this, for odd $N= 2m+1$, we have
\[
\fso(2m+1) \cong \bigwedge^2 \bV = \bigwedge^2 \bW \oplus (\bW \otimes \be) \oplus (\bW \otimes \bW^\ast) \oplus (\bW^\ast \otimes \be) \oplus \bigwedge^2 \bW^\ast.
\]
Here we use the standard decomposition of $\bigwedge^2 (\bW \oplus \bW^\ast \oplus \be)$. For even $N = 2m$, as discussed the standard representation is $\ol{\bV} = \bW \oplus \bW^\ast$, so we have
\[
\fso(2m) \cong \bigwedge^2 \ol{\bV} = \bigwedge^2 \bW \oplus (\bW \otimes \bW^\ast) \oplus \bigwedge^2 \bW^\ast.
\]
Throughout this paper, we prove $\fso(N)$-equivariance by considering the actions of the above summands separately. We adopt the notation of \cite{sam2016infinite} and define elements of $\fso(N)$ as follows
\begin{itemize}
\item For $v,w \in \bW$ we let $x_{v,w} = v \wedge w$ and $x_v = v \otimes e$.
\item For $v \in \bW$ and $\lambda \in \bW^\ast$ we let $h_{v,\lambda} = v \otimes \lambda$.
\item For $\lambda,\mu \in \bW^\ast$ we let $y_{\lambda,\mu} = \lambda \wedge \mu$ and $y_\lambda = \lambda \otimes e$.
\end{itemize}
Now we define a map $\fso(N) \to \fgl(\bV)$.  Suppose $u \in W \subset V$, then
{\small \[
x_{v,w} u = 0, \quad x_v u = 0, \quad h_{v,\lambda} u = \lambda(u) v, \quad y_\lambda u = \lambda(u)\be, \quad y_{\lambda,\mu} u = \mu(u) \lambda - \lambda(u) \mu.
\]}
We define the action on $\eta \in \bW^\ast$ similarly:
{\small \[
x_{v,w} \eta = \eta(w) v - \eta(v) w, \quad x_v \eta = -\eta(v)\be, \quad h_{v,\lambda} \eta = -\eta(v) \lambda, \quad y_\lambda \eta = 0, \quad y_{\lambda,\mu} \eta = 0.
\]}
Finally, put
\[
x_{v,w} \be = 0, \qquad x_v \be = v, \qquad h_{v,\lambda} \be = 0, \qquad y_\lambda \be = -\lambda, \qquad y_{\lambda,\mu} \be = 0.
\]
As mentioned in \cite{sam2016infinite} this is a well defined representation of $\fso(N)$ that respects the orthogonal form on $\bV$ for odd dimension and $\ol{\bV}$ for even dimension.  With this action, if $N$ is odd, $\bV$ is the {\bf standard representation} of $\fso(N)$.  If $N$ is even, $\ol{\bV}$ is the {\bf standard representation} of $\fso(N)$.

We will adopt the following perspective on the spin representation $\Delta$ of $\fso(N)$ found in  \cite{sam2016infinite,fulton2013representation}. Given the decomposition of the standard representation as $\bV = \bW \oplus \bW^\ast \oplus \be$ or $\ol{\bV} = \bW \oplus \bW^\ast$ with $\bW$ $m$-dimensional, we put 
\[
\Delta = \bigwedge^\bullet \bW = \bigwedge^0 \bW \oplus \cdots \oplus \bigwedge^m \bW
\]
the exterior algebra on $\bW$.

As in \cite{sam2016infinite}, we define the following operators on $\Delta$.  For $v \in \bW$, let $X_v$ be the operator on $\Delta$ given by
\[
X_v(w) = v \wedge w.
\]
And for $\lambda \in \bW^\ast$ let $D_{\lambda}$ be the operator on $\Delta$ given by
\[
D_{\lambda}(v_1 \wedge \cdots \wedge v_n) = \sum_{i=1}^n (-1)^{i-1} \lambda(v_i) v_1 \wedge \cdots \wedge \widehat{v_i} \wedge \cdots \wedge v_n.
\]
Where $\widehat{v_i}$ means we omit this term from the product. With these definitions let
\[
H_{v,\lambda} = X_v D_\lambda.
\]
This is the usual action of an element $v \otimes \lambda \in \fgl(W)$ on $\Delta$. Finally, define the operator $D$ by
\[
D(v_1 \wedge \cdots \wedge v_n) = (-1)^n v_1\wedge \cdots \wedge v_n.
\]
Notice both the operators $X_v$ and $X_w$ supercommute for any $v,w \in W$ as do the operators $D_\lambda$ and $D_\mu$ for any $\lambda, \mu \in \bW^\ast$.  That is, $X_v X_w + X_w X_v = 0$; so reversing the order of composition results in a sign change.  We also have the following relationship between the operators $X_v$ and $D_\lambda$
\begin{equation} \label{DXCommute}
 X_v D_\lambda + D_\lambda X_v = \lambda(v).
\end{equation}
Finally, we note that $D$ supercommutes with both $X_v$ and $Y_\lambda$. Given these operators, we define a representation $\rho$ of $\fso(N)$ on $\Delta$ as follows
\begin{align*}
&\rho(x_{v,w}) = X_v X_w &\rho(x_v) = \tfrac{1}{\sqrt{2}} &X_v D & \rho(h_{v,\lambda}) = H_{v,\lambda} - \tfrac{1}{2} \lambda(v)\\
&\rho(y_\lambda) = \tfrac{1}{\sqrt{2}} D D_\lambda &\rho(y_{\lambda,\mu}) =  D_\lambda D_\mu.
\end{align*}
This is a well-defined representation.  In particular, the scalar $\tfrac{1}{\sqrt{2}}$ ensures that the action respects the Lie bracket. We call $\Delta$ the {\bf Spinor Representation} of $\fso(N)$.

\section{Spin-Brauer Diagram Algebra} \label{Diagram Algebra}
We follow the work of Brauer \cite{brauer1937algebras} and define the spin-Brauer diagram algebra $\SB_n(\delta,)$ as a purely combinatorial object.  In particular, we describe an associative multiplication structure. 
\begin{definition} \label{DatumDef}
For any parameter $\delta$ and positive integer $n$, a {\bf spin-Brauer diagram} consists of five parts $(U,U',\Gamma,\Gamma', f)$ where
\begin{itemize}
\item $U$ and $U'$ are subsets of $T = \{1,\dots,n\}$ and $T' = \{1',\dots,n'\}$ with a total order corresponding to the standard total order on $T$ and $T'$,
\item $\Gamma$ and $\Gamma'$ are partial matchings on $T \setminus U$ and $T' \setminus U'$ respectively so that $|T \setminus (U \cup V(\Gamma))| = |T' \setminus (U' \cup V(\Gamma'))|$, here $V(\Gamma)$ denotes the vertex set of the graph describing the partial matching $\Gamma$. Recall that a partial matching is a list of pairs of elements from a given set.
\item $f$ is a bijection $T \setminus (U \cup V(\Gamma)) \to T' \setminus (U' \cup V(\Gamma'))$.
\end{itemize}
We call $(U,U',\Gamma,\Gamma', f)$ the {\bf spin datum} for the spin-Brauer diagram.
\end{definition}

We can think of this spin datum as a diagram by creating two rows of $n$ vertices corresponding to $T$ and $T'$.  The row for $T$ will be on the top and $T'$ on the bottom.  We circle all the elements of $U$ and $U'$ and label them with their total order. We call these circled vertices {\bf isolated vertices}. We then place an edge between $x,y \in T \setminus U$ if $(x,y) \in \Gamma$ and similarly draw an edge between the pairs in $\Gamma'$. These edges are called {\bf arcs}.  Finally, connect $x \in T \setminus (U \cup V(\Gamma))$ to $f(x) \in  T' \setminus (U' \cup V(\Gamma'))$ with an edge. These edges are called {\bf through strings}. Consider the following example converting a spin datum to a diagram.

\begin{example} \label{DatumExample}
Let $n = 5$.  The spin datum $U = \{2,5\}$, $U' = \{1',4'\}$, $\Gamma = \{(1,3)\}$, $\Gamma' = \{(2',5')\}$ and $f$ defined by $f(4) = 3'$ corresponds to the following spin-Brauer diagram
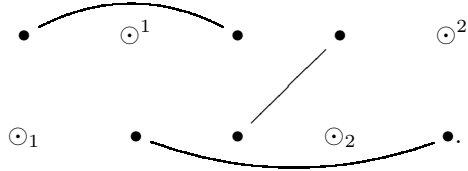
\[
\xymatrix{
\bullet \ar@{-}@/^1pc/[rr] & \odot^1 & \bullet &\bullet \ar@{-}[dl] & \odot^2\\
\odot_1 &\bullet \ar@{-}@/_1pc/[rrr] & \bullet & \odot_2 &\bullet.
}
\] 
Here $\Gamma$ and $\Gamma'$ describe the vertices connected by arcs.  $U$ and $U'$ describe the isolated vertices in the first and second row respectively. Then $f$ describes how the through strings connect. It is also clear that we could reverse this process and easily read off the spin datum from the diagram.
\end{example}

Fix a parameter $\delta$ and $n \in \bZ_{\geq 0}$. Let $\Omega_1 = (U_1,U_1',\Gamma_1,\Gamma_1',f_1)$ and $\Omega_2=(U_2,U_2',\Gamma_2,\Gamma_2',f_2)$ be spin-Brauer diagrams with $n$ vertices in each row.  We define a multiplication structure on spin-Brauer diagrams and extend this multiplication structure linearly. 
 Let $\Omega = \Omega_2\Omega_1$ be the diagram constructed as follows:
%
\begin{enumerate}[1)]
\item Place $\Omega_1$ on top of $\Omega_2$.
\item \label{Mult2} Define a new total ordering on all of the isolated vertices in $\Omega_1$ and $\Omega_2$ as $U_1 < U_1' < U_2 < U_2'$ with the total orders on each of the $U_i$ preserved.
\item \label{Mult3} Any arcs and isolated vertices in the top row of $\Omega_1$ are also in the top row of $\Omega$ and similarly any arcs and isolated vertices in the bottom row of $\Omega_2$ are added to the bottom row of $\Omega$. 
\item \label{Mult4} If we can follow any through strings from the top row of $\Omega_1$ in position $i$ to the bottom row of $\Omega_2$ in position $j$ draw a through string in $\Omega$ from vertex $i$ in the top row to vertex $j$ in the bottom row. Following a through string means walking along the path created when you place $\Omega_1$ on top of $\Omega_2$ that begins with vertex $i$.
\item \label{Mult5} If you follow a through string originating at a vertex $i$ in $\Omega_1$ to an isolated vertex with label $n_i$ in the total order, then vertex $i$ in the top row of $\Omega$ is isolated with label $n_i$. Similarly, if you follow a through string originating at a vertex $j$ in the bottom row of $\Omega_2$ to an isolated vertex with label $n_j$, the vertex $j$ in the bottom row of $\Omega$ is isolated with label $n_j$. 

\begin{remark}
Notice the resulting diagram will not be a spin-Brauer diagram because the vertices will not be in the correct total order. We will discuss the combinatorial rule for placing these vertices back into their totally ordered state below, we call this rule the spin-Clifford relation.
\end{remark}

After this step, all the vertices in the top and bottom row of $\Omega$ will be fixed. We define a {\bf closed circuit} in the product $\Omega_2\Omega_1$ as a connected component in the graph created by identifying the vertices in the bottom row of $\Omega_1$ and top row of $\Omega_2$.  As with the Brauer diagrams, closed circuits will scale the diagram $\Omega$ by a factor of $\delta$. We describe by example all the types of closed circuits that can occur in $\Omega$:

{\tiny
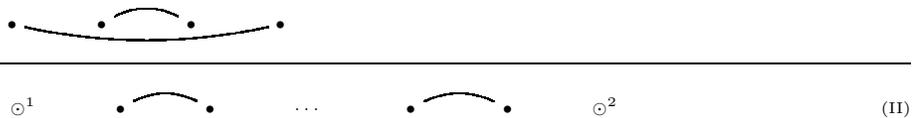
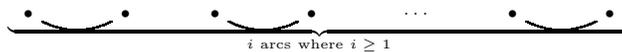
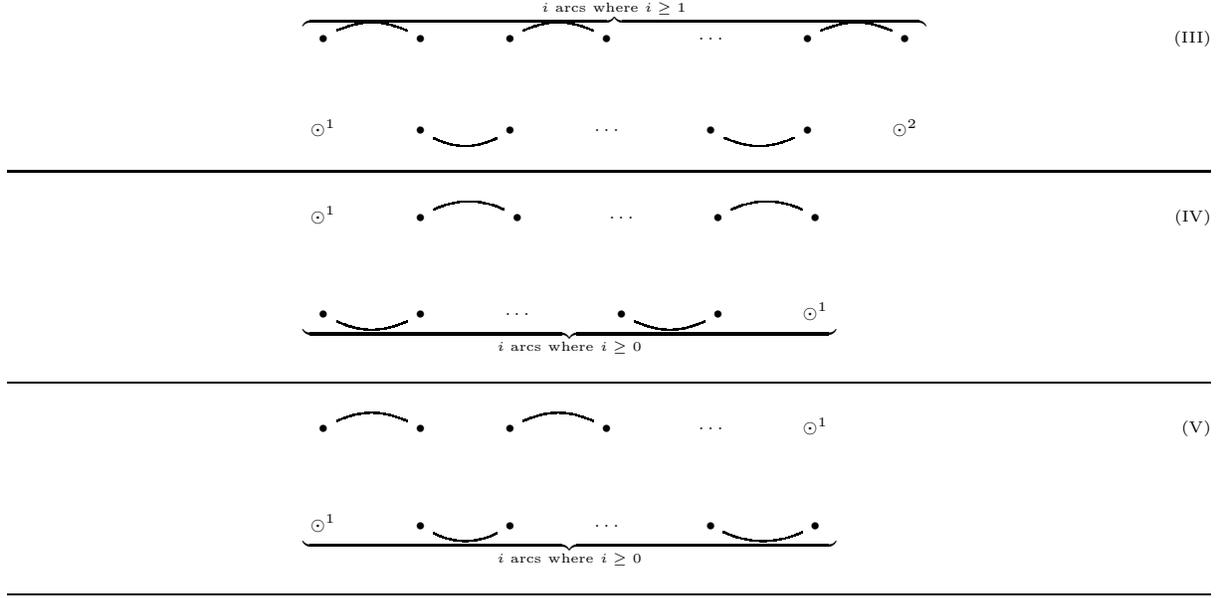
\begin{align*}
&\xymatrix{
\bullet \ar@{-}@/^/[rr] &\bullet \ar@{-}@/_/[rr] &\bullet &\bullet\\
\bullet \ar@{-}@/_/[rrr] &\bullet \ar@{-}@/^/[r] &\bullet &\bullet}
\tag{I} \label{Circ I}\\
\ \\
\hline
\ \\
&\underbrace{\xymatrix{
\odot^1 & \bullet \ar@{-}@/^/[r] & \bullet &\cdots & \bullet \ar@{-}@/^/[r] & \bullet &\odot^2  \\
\bullet \ar@{-}@/_/[r] &\bullet &\bullet \ar@{-}@/_/[r] &\bullet &\cdots &\bullet \ar@{-}@/_/[r] &\bullet}}_{\textrm{$i$ arcs where $i \geq 1$}}
\tag{II}\label{Circ II}\\
\ \\ 
\hline
\end{align*}
\begin{align*}
&\overbrace{\xymatrix{
\bullet \ar@{-}@/^/[r] &\bullet &\bullet \ar@{-}@/^/[r] &\bullet &\cdots &\bullet \ar@{-}@/^/[r] &\bullet\\
\odot^1 & \bullet \ar@{-}@/_/[r] & \bullet &\cdots & \bullet \ar@{-}@/_/[r] & \bullet &\odot^2}}^{\textrm{$i$ arcs where $i \geq 1$}}
 \tag{III} \label{Circ III}\\
\ \\ 
\hline
\ \\
&\underbrace{\xymatrix{
\odot^1 &\bullet \ar@{-}@/^/[r] &\bullet &\cdots &\bullet \ar@{-}@/^/[r] &\bullet\\
\bullet \ar@{-}@/_/[r] & \bullet &\cdots & \bullet \ar@{-}@/_/[r] & \bullet &\odot^1}}_{\textrm{$i$ arcs where $i \geq 0$}}
 \tag{IV} \label{Circ IV}\\
\ \\ 
\hline
\ \\
&\underbrace{\xymatrix{
\bullet \ar@{-}@/^/[r] &\bullet &\bullet \ar@{-}@/^/[r] &\bullet &\cdots &\odot^1\\
\odot^1 & \bullet \ar@{-}@/_/[r] & \bullet &\cdots & \bullet \ar@{-}@/_/[r] & \bullet}}_{\textrm{$i$ arcs where $i \geq 0$}}
 \tag{V} \label{Circ V}\\
\ \\ 
\hline
\end{align*}}

\item \label{Mult6} Scale the diagram $\Omega$ by $\delta$ for each closed circuit.
\item \label{Mult7} Let $U \sqcup U'$ be the indices of the isolated vertices as they appear in $\Omega$. Reindex the remaining isolated vertices preserving the total order induced from step (2) so that if $n$ index numbers appear, they are $\{1,\dots,n\}$.
\ \\

This accounts for any isolated vertices we may have removed in closed circuits.  At this step the isolated vertices are not in the correct total order. To fix this, we must generalize the Clifford-relation discussed in \cite[\S 2.3]{sam2016infinite}.  Suppose $\Omega$ is a diagram resulting from this multiplication process with $U = \{n_1,\dots,n_k\}$ and $U' = \{n_1',\dots,n_\ell'\}$ placed in the the total order from step \eqref{Mult7}.  Furthermore, let $\Omega'$ be obtained by switching two consecutive elements $i,j \in U \sqcup U'$ and letting the new $U$ be the first $k$ elements and the new $U'$ the last $\ell$.  Let $\Omega''$ be the diagram obtained by removing $i$ and $j$ from $U \sqcup U'$ and placing an edge between them. 

\begin{definition} \label{CliffordDef}
 The {\bf spin-Clifford relation} is $\Omega + \Omega' = 2\Omega''$.
\end{definition}

\begin{remark}
Notice the spin-Clifford relation is a strict generalization of the Clifford relation \cite{sam2016infinite}.  Indeed, we can swap isolated vertex indices within rows or across rows so long as they are consecutive. Furthermore, it applies to any diagram that arises from this multiplication process, not strictly spin-Brauer diagrams. We will see in Section \ref{Multiplication Agrees} that this generalization is possible because diagrams arising from this multiplication process correspond to well defined compositions of maps in the centralizer algebra. 
\end{remark}

\item\label{Mult8} Use the spin-Clifford relation to place the isolated vertices back in increasing order, where the total order comes from step \eqref{Mult2}. Do this in the minimal number of steps. If there are multiple minimal methods, choose the one which swaps the smallest possible pairs of vertices across rows, then corrects within rows.
\item \label{Mult9} Reindex the isolated vertices in the bottom row so they begin at one in the total order.

\end{enumerate}

Consider the following example of multiplying two spin-Brauer diagrams

\begin{example}
{\tiny
\[
\xymatrix @R=.7pc{
\bullet \ar@{-}[dd] &\bullet \ar@{-}@/^/[r] &\bullet &\odot^1 &\bullet \ar@{-}[ddr] &\bullet \ar@{-}@/^/[r] & \bullet\\
\\
\bullet &\odot_1 &\bullet \ar@{-}[r] &\bullet & \odot_2 &\bullet &\odot_3\\
\bullet \ar@{-}[dd] & \bullet \ar@{-}[r] &\bullet &\bullet \ar@{-}[r] &\bullet &\odot_1 &\bullet \ar@{-}[ddlllll]\\
\\
\bullet &\bullet &\bullet \ar@{-}@/_/[r] &\bullet & \odot_1 &\bullet \ar@{-}@/_/[r] &\bullet
}
\]}
This example will illustrate every piece of the multiplication process. Following step \eqref{Mult2}, we relabel all the isolated vertices to put them in a total order, increasing left to right and top to bottom. We preserve all of the arcs and isolated vertices as described in \eqref{Mult3}.  Similarly, we draw through strings as described in \eqref{Mult4}.  In this example, we have two through strings that terminate in isolated vertices.  As described in step \eqref{Mult5}, the originating vertex of each of these through strings becomes isolated.  Furthermore, according to step \eqref{Mult6} each closed circuit scales by $\delta$. We have one closed circuit in this example of type \eqref{Circ II} with $2$ arcs.  Applying all of these steps the first simplification is,
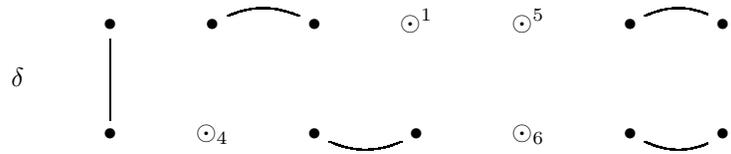
\[
\xymatrix@R=.6pc{
&\bullet \ar@{-}[dd] &\bullet \ar@{-}@/^/[r] &\bullet &\odot^1 &\odot^5 &\bullet \ar@{-}@/^/[r] & \bullet\;\\
\delta &&&&&&&\\
&\bullet &\odot_4 &\bullet \ar@{-}@/_/[r] &\bullet & \odot_6 &\bullet \ar@{-}@/_/[r] &\bullet.
}
\]
As mentioned, this is not a spin-Brauer diagram. Following step \eqref{Mult7} we reindex the isolated vertices while maintaining the order,
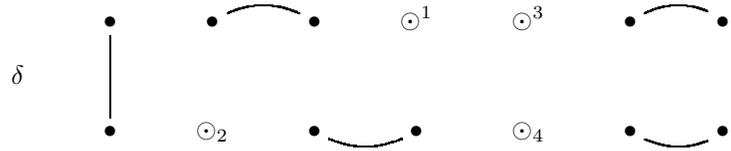
\[
\xymatrix@R=.6pc{
&\bullet \ar@{-}[dd] &\bullet \ar@{-}@/^/[r] &\bullet &\odot^1 &\odot^3 &\bullet \ar@{-}@/^/[r] & \bullet\;\\
\delta &&&&&&&\\
&\bullet &\odot_2 &\bullet \ar@{-}@/_/[r] &\bullet & \odot_4 &\bullet \ar@{-}@/_/[r] &\bullet.
}
\]
Again, this is not a spin-Brauer diagram.  The isolated vertices of this diagram currently have total order $1 < 3 < 2 < 4$. To fix this, we need the isolated vertices to be strictly increasing from left to right and top to bottom. Following the instructions in step \eqref{Mult8}, we swap the vertices labeled $2$ and $3$ in the total order via the spin-Clifford relation,
{\tiny
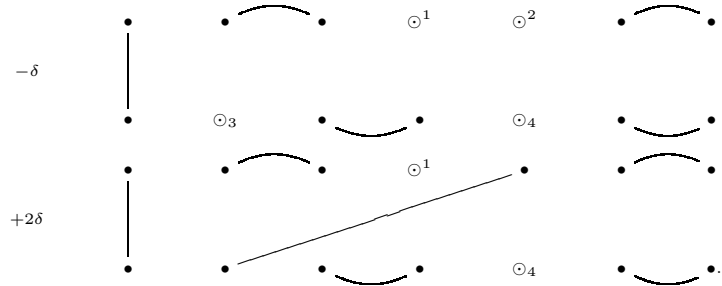
\begin{align*}
&\xymatrix@R=.6pc{
&\bullet \ar@{-}[dd] &\bullet \ar@{-}@/^/[r] &\bullet &\odot^1 &\odot^2 &\bullet \ar@{-}@/^/[r] & \bullet\;\\
-\delta &&&&&&&\\
&\bullet &\odot_3 &\bullet \ar@{-}@/_/[r] &\bullet & \odot_4 &\bullet \ar@{-}@/_/[r] &\bullet\;\\
&\bullet \ar@{-}[dd] &\bullet \ar@{-}@/^/[r] & \bullet &\odot^1 &\bullet \ar@{-}[ddlll] &\bullet \ar@{-}@/^/[r] & \bullet\;\\
+2\delta&&&&&&&\\
&\bullet &\bullet &\bullet \ar@{-}@/_/[r] &\bullet & \odot_4 &\bullet \ar@{-}@/_/[r] &\bullet.
}
\end{align*}}
Concluding, we proceed to \eqref{Mult9}, reindexing the isolated vertices to get,
{\tiny
\begin{align*}
&\xymatrix@R=.6pc{
&\bullet \ar@{-}[dd] &\bullet \ar@{-}@/^/[r] &\bullet &\odot^1 &\odot^2 &\bullet \ar@{-}@/^/[r] & \bullet\;\\
-\delta &&&&&&&\\
&\bullet &\odot_1 &\bullet \ar@{-}@/_/[r] &\bullet & \odot_2 &\bullet \ar@{-}@/_/[r] &\bullet\;\\
&\bullet \ar@{-}[dd] &\bullet \ar@{-}@/^/[r] & \bullet &\odot^1 &\bullet \ar@{-}[ddlll] &\bullet \ar@{-}@/^/[r] & \bullet\;\\
+2\delta&&&&&&&\\
&\bullet &\bullet &\bullet \ar@{-}@/_/[r] &\bullet & \odot_1 &\bullet \ar@{-}@/_/[r] &\bullet.
}
\end{align*}
}
\end{example}

\begin{definition}
For any ring $R$, parameter $\delta$ and $n \in \bZ_{\geq 0}$, define the algebra consisting of $R$-linear combinations of spin-Brauer diagrams on $2n$-vertices with the above multiplication structure as the {\bf spin-Brauer Diagram Algebra} denoted $\SB_n(\delta)$.
\end{definition}

\begin{definition}
We define $\cB(\SB_n(\delta))$ to be the basis for $\SB_n(\delta)$ consisting of all spin-Brauer diagrams on $n$ vertices.
\end{definition}

\begin{theorem}
For any ring $R$, $n \in \bZ_{\geq 0}$ and parameter $\delta$, $\SB_n(\delta)$ is an associative algebra. Furthermore, the multiplication coefficients for the basis $\cB(\SB_n(\delta)$ are in $\bZ[\delta]$.
\end{theorem}

\begin{proof}
It is clear from the multiplication construction that all the multiplication coefficients will be polynomials in $\delta$ with integer coefficients.

Associativity will follow from Theorems \ref{Iso} and \ref{THMMultAgrees} because for $\delta = N \geq 2n$ we will see that $\SB_n(\delta)$ is isomorphic to the centralizer algebra $\End_{\Pin(N)}(\bV^{\otimes n} \otimes \Delta)$ where composition is associative and hence multiplication in $\SB_n(N)$ will be associative in these cases, but this implies associativity for general $\delta$ as well.
%
\end{proof}

\section{Equivariant Projection and Injection Maps} \label{ProjInjMaps}
We now define a $\Pin(N)$-equivariant projection $\bV \otimes \Delta \twoheadrightarrow \Delta$ and injection $\Delta \injects \bV \otimes \Delta$. This explicit construction is the key to linking multiplication in $\SB_n(N)$ with the composition of the corresponding maps. In the following we will use the notation developed in \S\ref{Spinor Rep}.

\begin{definition}
Define the map $\pi\colon \bV \otimes \Delta \to \Delta$ on basis elements by
\[
\pi(v \otimes x) \coloneqq \sqrt{2} X_v(x) \qquad \pi(\lambda \otimes x) \coloneqq \sqrt{2} D_\lambda(x) \qquad \pi(\be \otimes x) \coloneqq D(x).
\]
Extend by linearity to all of $\bV \otimes \Delta$. We call $\pi$ the {\bf spin projection}. We also let $\pi_i$ for $1 \leq i \leq n$ denote the unique projection obtained by applying $\pi$ to the copy of $V \otimes \Delta$ sitting in tensor positions $i$ and $n+1$ and $\pi_i$ equal to the identity on all other tensor positions.
\end{definition}

\begin{lemma} \label{ProjEquivariance}
The spin projection, $\pi$, is a $\Pin(N)$-equivariant projection map.
\end{lemma}

\begin{proof}
It is clear that this is a surjective map. By bilinearity of the tensor product, linearity of $\pi$ and linearity of the action of $\fso(N)$, it suffices to prove equivariance on basis elements of $\Delta$ under the action of our basis for $\fso(N)$. We will suppress many of the calculations because they are straightforward and not particularly illuminating.

In \S \ref{Spinor Rep} we decomposed $\fso(N)$ as
\[
\bigwedge^2 \bW \oplus (\bW \otimes \be) \oplus (\bW \otimes \bW^\ast) \oplus (\bW^\ast \otimes \be) \oplus \bigwedge^2 \bW^\ast.
\]
We prove equivariance under the action of each of the summands. First assume $\dim(\bV) = 2m+1$ is odd.  We will see the even case is naturally contained in the odd case. \\
\\
It suffices to prove equivariance with respect to $w_i \wedge w_j \in \bigwedge^2 \bW$ when we map $w_k \otimes a$, $w_\ell^\ast \otimes a$ and $\be \otimes a$ where $a \in \Delta$ is arbitrary. Consider the action on $w_k \otimes a$,
\begin{align*}
\pi((w_i \wedge w_j) \cdot w_k \otimes a) &= \pi(w_k \otimes X_{w_i}X_{w_j}(a)) \\
& = (\sqrt{2}) X_{w_k} X_{w_i} X_{w_j}(a)\\
&= (w_i \wedge w_j) \cdot \pi(w_k \otimes a).
\end{align*}
This holds because our operators super-commute.  Now, when we consider the action on $w_\ell^\ast \otimes a$ we have
\begin{align*}
\pi((w_i \wedge w_j) \cdot w_\ell^\ast \otimes a) &= \pi\left((w_\ell^\ast(w_j)w_i - w_\ell^\ast(w_i) w_j) \otimes a + w_\ell^\ast \otimes X_{w_i}X_{w_j}(a)\right)\\
&= (w_i \wedge w_j) \cdot \pi( w_\ell^\ast \otimes a).
\end{align*}
Finally, if we consider the action on $\be \otimes a$, we see
\begin{align*}
\pi((w_i \wedge w_j) \cdot \be \otimes a) &= \pi(\be \otimes X_{w_i} X_{w_j}(a))\\
&= X_{w_i}X_{w_j} D(a)\\
&= (w_i \wedge w_j) \cdot \pi(e\otimes a).
\end{align*}
Here, we recall that the linear operator $D$ super-commutes with all other linear operators. This establishes equivariance for the $\bigwedge^2 \bW$ summand.  The verification process for $\bigwedge^2 \bW^\ast$ is similar so we leave it to the reader.

Similarly, we work through the case of $\bW^\ast \otimes \be$ and leave $\bW \otimes \be$-equivariance to the reader.  Consider $w_j^\ast \otimes \be$, a basis element of $\bW^\ast \otimes \be$.  As before, we first consider the action of $w_j^\ast \otimes \be$ on $w_i \otimes a \in \bV \otimes \Delta$.  We have
\begin{align*}
\pi((w_j^\ast \otimes \be) \cdot w_i \otimes a) &= \pi(w_j^\ast(w_i) \be \otimes a + w_i \otimes \tfrac{1}{\sqrt{2}}D D_{w_j^\ast}(a)) \\
&= DD_{w_j^\ast} X_{w_i}(a)\\
&= (w_j^\ast \otimes \be) \cdot \pi(w_i \otimes a).
\end{align*}
The main step follows from \eqref{DXCommute}. Now consider the action on the element $w_\ell^\ast \otimes a$,
\begin{align*}
\pi((w_j^\ast \otimes \be) \cdot w_\ell^\ast \otimes a) &= \pi(w_\ell^\ast \otimes \tfrac{1}{\sqrt{2}}DD_{w_j^\ast}(a))\\
&= DD_{w_j^\ast} D_{w_\ell^\ast}(a)\\
&= (w_j^\ast \otimes \be) \cdot \pi(w_\ell^\ast \otimes a).
\end{align*}
This follows from the skew commutativity of the operators. Finally consider the action on the element $e \otimes a$,
\begin{align*}
\pi((w_j^\ast \otimes \be) \cdot \be \otimes a) &= \pi(-w_j^\ast \otimes a + \be \otimes \tfrac{1}{\sqrt{2}}DD_{w_j^\ast}(a)) \\
&= -\sqrt{2} D_{w_j^\ast}(a) + \tfrac{1}{\sqrt{2}} D_{w_j^\ast}(a)\\
&= (w_j^\ast \otimes \be) \cdot \pi(\be \otimes a).
\end{align*}
As before, we use skew commutativity of the operators to obtain the third equality. This proves $\bW^\ast \otimes \be$-equivariance.  $\bW \otimes \be$-equivariance is similar and so we leave it to the reader. It remains to prove $\fgl(\bW)$ equivariance, but this is clear from the definition of our maps. 

If $N$ is even, we check equivariance for $\bigwedge^2 \bW$, $\fgl(\bW)$ and $\bigwedge^2 \bW^\ast$ exactly as above but now we do not have to consider the action of $\be \otimes a$.  Hence, the even case is naturally contained in the above work. As $\pi$ is equivariant with respect to each of the summands, it is $\fso(N)$-equivariant for any positive integer $N$. 

This establishes $\fso(N)$-equivariance and hence $\Spin(N)$-equivariance.  As discussed in \S \ref{introduction}, to prove $\Pin(N)$-equivariance, it suffices to prove $\pi$ is equivariant under the action of $\tfrac{1}{\sqrt{2}} (w_1 - w_1^\ast) \in \Pin(N) \cap \fC^{\text{odd}}(\omega)$.

\begin{remark}
We note that $\tfrac{1}{\sqrt{2}} (w_1 - w_1^\ast) \in \Pin(N) \cap \fC^{\text{odd}}(\omega)$.  Clearly it is in the odd part of the Clifford algebra.  To see it is in $\Pin(N)$ notice
\[
\omega\left(\tfrac{1}{\sqrt{2}} (w_1 - w_1^\ast),\tfrac{1}{\sqrt{2}} (w_1 - w_1^\ast)\right) = -1.
\]
\end{remark}

We check equivariance when acting on $w_i \otimes a$.  For ease of notation, let $\gamma = \tfrac{1}{\sqrt{2}} (w_1 - w_1^\ast)$. For explicit descriptions of the representations $\bV$ and $\Delta$ of $\Pin(N)$, we refer the reader to Fulton-Harris \cite[\S 20]{fulton2013representation}. We do note that in the definition of our bilinear form, we do not scale by $2$ like Fulton-Harris. Accordingly, the element $w_i^* \in \Pin(N)$ acts by $D_{w_i^*}$, not $2D_{w_i^*}$. When we consider the action of $\gamma$ on $w_i \otimes a$, we see
\begin{align*}
\pi(\gamma \cdot (w_i \otimes a)) &= \pi(\gamma \cdot w_i \otimes \gamma \cdot a)\\
&= \gamma \cdot (\tfrac{1}{\sqrt{2}} X_{w_i}(a))\\
&= \gamma \cdot \pi(w_i \otimes a).
\end{align*}
Checking equivariance for $w_i^* \otimes a$ and $\be \otimes a$ are similar and so are left to the reader.  This proves $\Pin(N)$-equivariance. \qed
\end{proof}

\begin{definition} When $\bV$ is even dimensional, with $\dim(\bV) = 2m$ define the map $\iota\colon \Delta \to \bV \otimes \Delta$ by
\[
\iota(a) \coloneqq \sqrt{2} \left( \sum_{i=1}^m w_i \otimes D_{w_i^\ast}(a) + w_i^\ast \otimes X_{w_i}(a) \right).
\]
When $\bV$ is odd dimensional, with $\dim(\bV) = 2m+1$ define
\[
\iota(a) \coloneqq \sqrt{2} \left( \sum_{i=1}^m w_i \otimes D_{w_i^\ast}(a) + w_i^\ast \otimes X_{w_i}(a)\right) + e \otimes D(a).
\]
We call $\iota$ the {\bf spin injection}. We also let $\iota_j$ for $1 \leq j \leq n+1$ denote the unique injection obtained by applying $\iota$ to the copy of $\Delta$ in $\bV^{\otimes n} \otimes \Delta$ and placing the resulting copy of $\bV$ into the $j$th tensor position of $\bV^{\otimes n+1}$ with the other tensor positions shifted accordingly. On all other tensor positions $\iota_j$ is the identity.
\end{definition}

\begin{lemma} \label{InjEquivariance}
The spin injection, $\iota$, is a $\Pin(N)$-equivariant injection.
\end{lemma}

\begin{proof}
It is clear that this is an injection. We will proceed as in the proof of Lemma \ref{ProjEquivariance}. First assume that $\dim(\bV) = 2m+1$ is odd.  We will prove equivariance in this case and deduce equivariance in the even dimensional case. We will suppress many of the calculations because they are straightforward and not particularly illuminating. Let $w_j \wedge w_k \in \bigwedge^2 \bW$ and $a \in \Delta$, we have
\begin{align*}
\iota((w_j \wedge w_k) \cdot a) &= \iota(X_{w_j}X_{w_k}(a))\\
&= (w_j \wedge w_k) \cdot \left[\sqrt{2} \cdot \sum_{i=1}^m w_i \otimes D_{w_i^\ast}(a) + w_i^{\ast} \otimes X_{w_i}(a) + \be \otimes D(a) \right]\\
&= (w_j \wedge w_k) \cdot \iota(a).
\end{align*}
The case for $w_j^\ast \wedge w_k^\ast$ is similar and left to the reader.  Now consider the action of the element $w_j^\ast \otimes \be$,
\begin{align*}
\iota((w_j^\ast \otimes \be) \cdot a) &= \iota(\tfrac{1}{\sqrt{2}}DD_{w_j^\ast}(a))\\
&= \sum_{i=1}^m w_i \otimes D_{w_i^\ast}DD_{w_j^\ast}(a) + w_i^\ast \otimes X_{w_i} DD_{w_j^\ast}(a) + \tfrac{1}{\sqrt{2}} \be \otimes D_{w_j^\ast}(a)\\
&= (w_j^\ast \otimes \be) \cdot \iota(a).
\end{align*}
The case for $w_j \otimes \be$ is similar and left to the reader.  It remains to verify $\fgl(W)$-equivariance, which is straightforward from the construction of our maps. 

As in the proof of Lemma \ref{ProjEquivariance}, the even dimensional case is contained in the above.  Thus $\iota$ is $\fso(N)$-equivariant and so $\Spin(N)$-equivariant.  We now check that $\iota$ commutes with the action of $\gamma = \tfrac{1}{\sqrt{2}} (w_1 - w_1^*)$,
\begin{align*}
\iota(\gamma \cdot a) &= \iota( \tfrac{1}{\sqrt{2}} X_{w_1}(a) - \tfrac{1}{\sqrt{2}} D_{w_1^*}(a))\\
&= \sum_{i=1}^m w_i \otimes D_{w_i^*} (X_{w_1} - D_{w_1^*})(a) + w_i^* \otimes X_{w_i}(X_{w_1} - D_{w_1^*})(a)\\
& \qquad \qquad + \tfrac{1}{\sqrt{2}} \be \otimes D(X_{w_1} - D_{w_1^*})(a)\\
&= \gamma \cdot \iota(a).
\end{align*}
This proves that $\iota$ commutes with the action of $\gamma$ which by the discussion in \S\ref{introduction} implies $\Pin(N)$-equivariance. \qed
\end{proof} 

For $\sigma \in \fS_n$, let $\tau_{\sigma}$ be the equivariant map $\bV^{\otimes n} \otimes \Delta \to \bV^{\otimes n} \otimes \Delta$ sending the vector in tensor position $i$ to tensor position $\sigma(i)$.  We call $\tau_{\sigma}$ the {\bf swap operator}.  

We also have an equivariant map $\psi_{i,j}$ from \cite[\S 8]{koike2005spin}. Let $\psi_{i,j}$ be the linear immersion of the invariant element
\[
\sum_{i=1}^m w_i \otimes w_i^\ast + w_i^\ast \otimes w_i + (\be \otimes \be)
\]
into the $i,j$ tensor positions of $\bV^{\otimes n} \otimes \Delta$ where $\dim(\bV) = 2m+1$.  If $\dim(\bV) = 2m$, the $\fso(N)$-invariant element is
\[
\sum_{i=1}^m w_i \otimes w_i^\ast + w_i^\ast \otimes w_i.
\]
To realize the spin-Brauer diagrams as elements of the centralizer algebra, we need an additional equivariant map $\kappa_{i,j}$ for $1 \leq i < j \leq n$ called the {\bf contraction}.  This map is given by contracting the elements in the $i^{th}$ and $j^{th}$ tensor positions of $\bV^{\otimes n} \otimes \Delta$ using the bilinear form on $\bV$. By construction, $\fso(N)$ respects this bilinear form so this is an equivariant map.

\begin{remark}
Any operators that act on different tensor positions commute. This is clear, but we point it out because it will be important.
\end{remark}

\section{Spin-Brauer Multiplication agrees with Composition} \label{Multiplication Agrees}
In this section we discuss the correspondence between $\SB_n(N)$ and maps in $\End_{\Pin(N)}(\bV^{\otimes n} \otimes \Delta)$.  We conclude this section by proving that our combinatorial description of multiplication agrees with the corresponding composition of maps.

Let $\Omega \in \cB(\SB_n(N))$ be a spin-Brauer diagram with spin datum $\Omega= (U,U',\Gamma,\Gamma',f)$ and let $T = \{1,\dots,n\}$.  Furthermore, let $\sigma \in \fS_n$ be the permutation induced by $f$.  That is, $\sigma(i) = f(i)$ if $i \in T \setminus (U \cup V(\Gamma))$ and $\sigma(i) = i$ otherwise. Then $\Omega$ corresponds to the following equivariant map
\begin{equation} \label{mapconvert}
\Omega \mapsto \prod_{(i,j) \in \Gamma'} \psi_{i,j} \circ \prod_{j \in U'} \iota_j \circ \tau_\sigma \circ \prod_{i \in U} \pi_i \circ \prod_{(i,j) \in \Gamma} \kappa_{i,j} =: f_\Omega.
\end{equation}
Extend the correspondence in (\ref{mapconvert}) by linearity to all of $\SB_n(\delta)$.

\begin{theorem} \label{Iso}
Under this correspondence, for $n,N \in \bZ^+$, $\SB_n(N)$ surjects onto the centralizer algebra $\End_{\Pin(N)}(\bV^{\otimes n} \otimes \Delta)$ and for $N \geq 2n$ the map in (\ref{mapconvert}) is a bijection.
\end{theorem}

\begin{proof}
This follows from \cite[\S 5, \S7]{koike2005spin}.  Indeed, Koike gives formulas for decomposing his maps into a composition of projections, injections, immersions and contractions as in \eqref{mapconvert} \cite[\S6, Theorem 8.1]{koike2005spin}. Accordingly, it suffices to prove our maps on one tensor component agree up to a scalar as each map in $\End_{\Pin(N)}(\bV^{\otimes n} \otimes \Delta)$ is defined as a composition of these maps. Hence if each of the maps on one tensor component agree, then the composition will also agree.

From Lemmas \ref{ProjEquivariance} and \ref{InjEquivariance} the spin projection and injection are $\Pin(N)$-equivariant and thus $\fso(N)$-equivariant. Due to the semi-simplicity of $\fso(N)$, every finite dimensional representation can be decomposed as a direct sum of irreducible representations. In particular, we can view $\iota$ and $\pi$ as equivariant maps between the irreducible components of $\Delta$ and $\bV \otimes \Delta$. 

These maps are unique on each of the components up to a scalar by Schur's lemma for semi-simple Lie algebras. Accordingly, to show uniqueness of $\iota$ and $\pi$ it suffices to show $\Delta$ and $\bV \otimes \Delta$ decompose into a direct sum of irreducible representations with no multiplicities.

\begin{lemma} \label{Decompose}
$\Delta$ and $\bV \otimes \Delta$ decompose into a direct sum of irreducible representations with no multiplicities.
\end{lemma}

\begin{proof}
The spin representation $\Delta$ is irreducible when $\dim(\bV)$ is odd and is the direct sum of the two distinct irreducible half spin representations when $\dim(\bV)$ is even \cite[\S 20]{fulton2013representation}. So it clearly decomposes as a direct sum of irreducible representations with no multiplicities.

Let $L_i$ be the linear function on diagonal matrices, the Cartan subalgebra, whose output is the $i^{th}$ diagonal entry. This is defined completely in \cite[\S 12]{fulton2013representation}.

When $\dim(\bV) = 2n+1$ is odd, $\Delta$ is irreducible with highest weight $\tfrac{1}{2} (L_1 + \cdots + L_n)$. If $\dim(\bV) = 2n$ is even, $\Delta$ decomposes as a direct sum of two representations with highest weights $\tfrac{1}{2}(L_1+\cdots+L_n)$ and $\tfrac{1}{2}(L_1 + \cdots + L_{n-1} - L_n)$. For proof of these facts we refer the reader to \cite[\S 20.1]{fulton2013representation}. While $\bV$ has weights $\{\pm L_i\} \cup \{0\}$ defined in \cite[\S 18.1]{fulton2013representation}.

The weight diagram of $\bV \otimes \Delta$ is generated by $\alpha + \beta$ where $\alpha$ is a weight of $\bV$ and $\beta$ is a weight of $\Delta$. It is easy to check that the resulting weight diagram has no multiplicities. When $\dim(\bV)$ is odd it is trivial. When $\dim(\bV)$ is even suppose we had
\[
L_i + \tfrac{1}{2}(L_1+\cdots+L_n) = L_j + \tfrac{1}{2}(L_1 + \cdots + L_{n-1} - L_n),
\]
this is the only interesting case. Simplifying we see
\[
L_i = L_j - \tfrac{1}{2} L_n.
\]
This is impossible. This implies every irreducible representation in $\bV \otimes \Delta$ will occur without multiplicity. \qed
\end{proof}

We conclude from Lemma \ref{Decompose} that $\pi$ and $\iota$ are uniquely determined up to a scalar. This implies Koike's equivariant maps $\Delta \injects \bV \otimes \Delta$ and $\bV \otimes \Delta \surjects \Delta$ must agree with $\iota$ and $\pi$ up to a scalar.

Koike used invariant theory to prove that the image of his generalized Brauer diagrams span the centralizer algebra \cite[Lemma 5.6]{koike2005spin}. By the above, the images of our spin-Brauer diagrams must span the centralizer algebra as well.  \qed
\end{proof}

\begin{theorem} \label{THMMultAgrees}
If $n,N \in \bZ_{\geq 0}$, for any $\Omega_1,\Omega_2 \in \SB_n(N)$ we have $f_{\Omega_2\Omega_1} = f_{\Omega_2} \circ f_{\Omega_1}$.
\end{theorem}

\begin{proof}
By linearity it suffices to verify $f_{\Omega_2\Omega_1} = f_{\Omega_2} \circ f_{\Omega_1}$ for $\Omega_1,\Omega_2 \in \cB(\SB_n(N))$ and for an arbitrary basis element of $\bV^{\otimes n} \otimes \Delta$.

Let $\Omega_1,\Omega_2 \in \cB(\SB_n(N))$.  Notice, our multiplication construction agrees with Brauer's \cite{brauer1937algebras}.  That is, ${\bf B}_n(N)$ is a subalgebra of $\SB_n(N)$. As a result, we know our theorem holds for any components of the diagrams that appear in Brauer diagrams. We may therefore assume $\Omega_1$ and $\Omega_2$ do not have any through strings creating a path from the top row of $\Omega_1$ to the bottom row of $\Omega_2$.  Furthermore, we can assume there are no arcs that form closed circuit \eqref{Circ I}.
%
%

Now, if we consider the composition $f_{\Omega_2} \circ  f_{\Omega_1}$ we need to simplify the maps so that the composition is of the form (\ref{mapconvert}).  This will correspond to a sum of maps which we will show is $f_{\Omega_2 \Omega_1}$. 

Notice the spin projections and contractions in the first row of $\Omega_1$ must remain as do the spin injections and immersions in the bottom row of $\Omega_2$. Indeed, it suffices to compute what Koike calls the ``inside homomorphism".  That is, the compositions of maps between the spin contractions and projections of $\Omega_1$ and the spin injections and immersions of $\Omega_2$. These are the maps that resolve the bottom row of the top diagram and top row of the bottom diagram.  For further discussion we refer the reader to \cite[\S 9]{koike2005spin}.

By assumption, we can follow every through string originating from a vertex in the top row of $\Omega_1$ or bottom row of $\Omega_2$ to an isolated vertex, as we have shown all other parts of our multiplication agree with composition in $\End_{\Pin(N)}(\bV^{\otimes n} \otimes \Delta)$. By the above observations, to prove the theorem, it remains to prove:
\begin{enumerate}[(1)]
\item All the closed circuits correspond to scaling by $\dim(V)$. 
\item Through strings leading to isolated vertices become isolated.  
\end{enumerate}
These are precisely the remaining parts of our combinatorial multiplication that are not handled by the above observations. To prove part (2), we must also show the Spin-Clifford relation agrees with the corresponding composition of maps.  In summary, the theorem breaks down into a series of lemmas.

\begin{lemma}\label{CS2}
Closed circuit (\ref{Circ II}) corresponds to scaling by $N= \dim(\bV)$.
\end{lemma}

\begin{proof}
We are considering a closed circuit in our diagram, so no through strings will begin or end in any of our vertices.  Equivalently, after applying the first projections maps we project away all of the entries in these tensor positions. We will use this fact in all of the following lemmas. We keep track of entries we project away with a dash. For example, if we contract the first and third tensor position of $v_1 \otimes v_2 \otimes v_3 \otimes a$, we write $\omega(v_1,v_3) (- \otimes v_2 - \otimes a)$.

As noted, it suffices by linearity of our maps to prove this lemma for simple tensors. Furthermore, because we sum over all basis vectors, we can permute the indices in the sum corresponding to every closed circuit of the form \eqref{Circ II} so that the sum resembles the example given in \eqref{Circ II}. Accordingly, it suffices to prove the result for this example. Suppose our circuit has length $k < n$.  Then, circuit \eqref{Circ II} corresponds to the maps 
\[
\kappa_{k,k-1} \circ \cdots \circ \kappa_{1,2} \circ \psi_{k-2,k-1} \circ \cdots \circ \psi_{2,3} \circ \iota_k \circ \iota_1. 
\]
Let $-\otimes \cdots \otimes - \otimes a \in \bV^{\otimes k} \otimes \Delta$ with $a = w_{i_1} \wedge \cdots \wedge w_{i_\ell}$ a simple tensor in the basis for $\Delta$.  Here we only consider the $k$ tensor positions in $\bV^{\otimes n} \otimes \Delta$ involved in our closed circuit.  First, suppose $\dim(\bV) = 2m$.

When we apply all the injections we have
{\small
\begin{align*}
2 \sum_{\substack{i_1,i_2=1\\ j_1,\dots,j_{(k-2)/2=1}}}^m 
&
w_{i_1}^\ast \otimes w_{j_1} \otimes w_{j_1}^\ast \otimes \cdots \otimes w_{j_{(k-2)/2}} \otimes w_{j_{(k-2)/2}}^\ast \otimes w_{i_2} \otimes X_{w_{i_2}} \circ D_{w_{i_1}^\ast}(a)+\\
&
w_{i_1} \otimes w_{j_1}^\ast \otimes w_{j_1} \otimes \cdots \otimes w_{j_{(k-2)/2}}^\ast \otimes w_{j_{(k-2)/2}} \otimes w_{i_2}^\ast \otimes D_{w_{i_2}^*} \circ X_{w_{i_1}}(a)+\\
&\cdots .
\end{align*}}
There are far more terms in the sum corresponding to all the possible permutations of the spin immersions.  However, it suffices to consider the terms that alternate between elements of $\bW$ and $\bW^\ast$.  Indeed, $\bW$ and $\bW^\ast$ are isotropic, so when we apply the spin contraction the only pairs of basis elements that survive are $w_{i_1} \otimes w_{j_k}^\ast$ where $i_1 = j_k$.

When we apply the contraction map to tensor positions $1$ and $2$, it equates the indices in the sum. We record this by reindexing the sum from $i_1 = j_1=1$ to $m$.  We continue applying the contractions. Each time we apply a contraction, we equate two more indices. When we apply the final contraction map, we have equated all of the indices and projected every tensor position.  In the end, we have forced the string of equalities $i_1 = j_1 = j_2 = \cdots = j_{(k-2)/2} = i_2$. Our sum is now  
\begin{align*}
2\sum_{i=1}^m &-\otimes\cdots\otimes-\otimes X_{w_i} \circ D_{w_{i}^\ast}(a) +\\
&-\otimes\cdots\otimes-\otimes D_{w_i^\ast} \circ X_{w_{i}}(a).
\end{align*}
The map $X_{w_i} \circ D_{w_{i}^\ast}(a)$ will be the identity if $w_i$ is one of the components of $a \in \Delta$ and zero otherwise. On the other hand, $D_{w_i^\ast} \circ X_{w_{i}}(a)$ is the identity when $w_i$ is not a component of $a$ and zero otherwise.  As we sum over all basis vectors, we get one copy of $-\otimes\cdots\otimes-\otimes a$ for each index $i=1,\dots,m$.  This leaves us with
\begin{equation}\label{eq1CS2}
(2m)(-\otimes\cdots\otimes-\otimes a).
\end{equation}
This is precisely the identity map scaled by $2m = \dim(\bV)$.

If $\dim(\bV) = 2m+1$, we are in a similar situation. However, now we have additional terms corresponding to the spanning element $\be$. In particular, the immersions now contain an additional $\be \otimes \be$.  When we contract a term containing $\be$, if any other tensor position is not $\be$ the tensor vanishes. As a result, one additional term in the sum will be nonzero after the spin contractions. This is the term where every tensor position contains $\be$,
\[
 \be \otimes \be \otimes \cdots \otimes \be \otimes \be \otimes \be \otimes D\circ D(a).
\]
Notice that $D \circ D$ is the identity.  So when we apply the spin contraction to each of these positions what remains is
\[
- \otimes \cdots \otimes - \otimes a,
\]
one additional copy of our original tensor.  Adding this to \eqref{eq1CS2}, we see the closed circuit corresponds to the map
\[
- \otimes \cdots - \otimes a \mapsto (2m+1) (- \otimes \cdots - \otimes a).
\]
Once again, this is precisely the identity map scaled by $2m+1 = \dim(\bV)$. \qed
\end{proof}

\begin{lemma} \label{CS3}
Closed circuit (\ref{Circ III}) corresponds to scaling by $N= \dim(\bV)$.
\end{lemma}

\begin{proof}
We proceed as in Lemma \ref{CS2}. Suppose our circuit has length $k < n$.  We can permute the indices in the sum corresponding to every closed circuit of the form \eqref{Circ III} so that the sum equals the example given in \eqref{Circ III}. Accordingly, it suffices to prove this result for this example. Circuit \eqref{Circ III} corresponds to the map
\[
\pi_k \circ \pi _1 \circ \kappa_{2,3} \circ \cdots \circ \kappa_{k-2,k-1} \circ \psi_{k,k-1} \circ \cdots \circ \psi_{1,2}.
\]
After applying the immersions we have
{\small\begin{align*}
- \otimes \cdots \otimes - \otimes a \mapsto &\sum_{j_1,\dots,j_{k/2}=1}^m &w_{j_1} \otimes w_{j_1}^\ast \otimes w_{j_2} \otimes w_{j_2}^\ast \otimes \cdots \otimes w_{j_{k/2}} \otimes w_{j_{k/2}}^\ast \otimes a+\\
&& w_{j_1}^\ast \otimes w_{j_1} \otimes w_{j_2}^\ast \otimes w_{j_2} \otimes \cdots \otimes w_{j_{k/2}}^\ast \otimes w_{j_{k/2}} \otimes a+\\
&& \cdots .
\end{align*}}
The remaining terms will all vanish when we apply the spin contractions because some pair will contain two elements from $\bW$ or $\bW^\ast$. Applying the contractions forces equality among the indices, i.e. $j_1 = j_2 = \cdots = j_{k/2}$. This gives us
\begin{align*}
\sum_{i=1}^m &w_{i} \otimes - \otimes - \otimes - \otimes \cdots \otimes - \otimes w_{i}^\ast \otimes a+\\
& w_{i}^\ast \otimes - \otimes - \otimes - \otimes \cdots \otimes - \otimes w_{i} \otimes a.
\end{align*}
Now if we apply our spin projections we recognize the same sum from Lemma \ref{CS2},
\begin{align*}
2 \sum_{i=1}^m & - \otimes - \otimes - \otimes - \otimes \cdots \otimes - \otimes - \otimes D_{w_i^\ast} X_{w_i}(a)+\\
& - \otimes - \otimes - \otimes - \otimes \cdots \otimes - \otimes - \otimes X_{w_i} D_{w_i^\ast}(a).
\end{align*}
This sum is precisely $(2m)(-\otimes \cdots \otimes - \otimes a)$. So we scale by $\dim(\bV)$.  When $\dim(\bV) = 2m+1$ there will be one more term that does not vanish when we apply the spin contractions,
\[
\be \otimes \be \otimes \cdots \otimes \be \otimes a.
\]
After the contractions and spin projections this term is sent to
\[
- \otimes - \otimes \cdots \otimes - \otimes DD(a) = - \otimes - \otimes \cdots \otimes - \otimes a.
\]
Thus if $\dim(\bV) = 2m+1$ this closed circuit corresponds to the map sending  $-\otimes \cdots \otimes - \otimes a$ to $(2m+1)(-\otimes \cdots \otimes - \otimes a)$. \qed
\end{proof}

\begin{lemma} \label{CS4}
Closed circuit (\ref{Circ IV}) corresponds to scaling by $N= \dim(\bV)$.
\end{lemma}

\begin{proof}
We proceed as in Lemmas \ref{CS2} and \ref{CS3}.  Suppose our circuit has length $k < n$.  Then circuit \eqref{Circ IV} corresponds to the map
\[
\pi_k \circ \kappa_{1,2} \circ \cdots \circ \kappa_{k-2,k-1} \circ \psi_{2,3} \circ \cdots \circ \psi_{k-1,k} \circ \iota_1.
\]
Applying all the immersions,
{\small \begin{align*}
- \otimes \cdots \otimes - \otimes a \mapsto \sqrt{2} \sum_{i,j_1,\dots,j_{k-1/2}=1}^m &w_i \otimes w_{j_1}^\ast \otimes w_{j_1} \otimes \cdots \otimes w_{j_{k-1/2}}^\ast \otimes w_{j_{k-1/2}} \otimes D_{w_i^\ast}(a)+\\
& w_i^\ast \otimes w_{j_1} \otimes w_{j_1}^\ast \otimes \cdots \otimes w_{j_{k-1/2}} \otimes w_{j_{k-1/2}}^\ast\otimes X_{w_i}(a) +\\
&\be \otimes \be \otimes \be \otimes \cdots \otimes \be \otimes \be \otimes D(a).
\end{align*}}
Applying the contractions equates all of the indices.  Then when we apply the spin projection to the $k^{th}$ tensor position we have
\begin{align*}
2 \sum_{i=1}^m &- \otimes - \otimes \cdots \otimes - \otimes X_{w_i} D_{w_i^\ast}(a)+\\
& - \otimes - \otimes \cdots \otimes - \otimes D_{w_i^\ast}X_{w_i} (a)+\\
& - \otimes - \otimes \cdots \otimes - \otimes DD(a).
\end{align*}
This is the sum we saw in the previous two lemmas. Using the same reasoning, we can conclude this closed circuit results in scaling by $\dim(\bV)$. \qed
\end{proof}

\begin{lemma} \label{CS5}
Closed circuit (\ref{Circ V}) corresponds to scaling by $N= \dim(\bV)$.
\end{lemma}

\begin{proof}
Apply the same proof as in Lemma \ref{CS4}, but we now inject into the last tensor position and project from the first. \qed
\end{proof}

\begin{lemma} \label{TSLemma}
If we follow a through string to an isolated vertex this corresponds to replacing the originating vertex of the through string with the corresponding isolated vertex.
\end{lemma}

\begin{proof}
It suffices to consider two cases:
\begin{enumerate}[(1)] 
\item When the through string terminates directly in an isolated vertex.
\item When the through string is connected to an isolated vertex via one spin contraction.
\end{enumerate}
Indeed, if we travel along multiple immersions and contractions to reach an isolated vertex as in
{\tiny
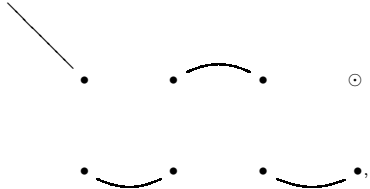
\[
\xymatrix{
&\\
&\bullet \ar@{-}[ul] & \bullet \ar@{-}@/^/[r] &\bullet &\odot\;\;\\
&\bullet \ar@{-}@/_/[r] &\bullet  &\bullet \ar@{-}@/_/[r] &\bullet ,
}
\]}
we equate all of the indices we introduced as in the proofs of Lemmas \ref{CS2}-\ref{CS5}. The corresponding sum is the same as the sum associated to the maps in the diagram
{\tiny
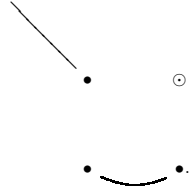
\[
\xymatrix{
&\\
&\bullet \ar@{-}[ul]  &\odot\;\\
&\bullet \ar@{-}@/_/[r] &\bullet.
}
\]}
The same reasoning applies for any diagram with the terminal isolated vertex in the second row.  In this case, the sum corresponding to these diagrams reduces to the sum associated to
{\tiny
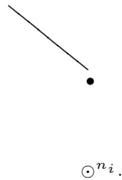
\[
\xymatrix{
&\\
&\bullet \ar@{-}[ul]\;\;\;\; \\
&\odot^{n_i}.
}
\]}
Consider a diagram of this type.  The isolated vertex corresponds to $n_i$ in the total order. Suppose the through string originates in tensor position $k$ and terminates in tensor position $\ell$.  Clearly, the following two operations are equivalent
\begin{itemize}
\item Send a vector $v$ in tensor position $k$ to tensor position $\ell$ then apply a spin projection to position $\ell$ in the appropriate order corresponding to $n_i$.
\item Project the vector $v$ from tensor position $k$ as the $n_i^{th}$ projection.  
\end{itemize}
This corresponds to replacing the origin of the through string with the isolated vertex of index $n_i$. Now consider the second case.  That is,
{\tiny
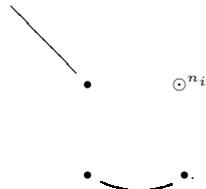
\[
\xymatrix{
&\\
&\bullet \ar@{-}[ul]  &\odot^{n_i}\\
&\bullet \ar@{-}@/_/[r] &\bullet.\;
}
\]}
This is a slightly more interesting case because a spin immersion becomes a spin projection. To see how this occurs, suppose the through string originates in tensor position $k$ and terminates in position $\ell$.  When we consider the corresponding maps, we have
\[
v \otimes - \otimes - \otimes a \mapsto - \otimes v \otimes - \otimes a.
\]
Without loss of generality, we assume the first tensor position is position $k$, the second is $\ell$ and the last is $\ell + 1$. We can do this because the maps we consider only affect these tensor positions.  Furthermore, any application of linear operators in the spin representation tensor position occur consecutively, so we may isolate these maps. When we apply the spin injection we have
\begin{equation} \label{TSEQ1}
\sqrt{2} \left[ \sum_{i=1}^m - \otimes v \otimes w_i \otimes D_{w_i^\ast}(a) + - \otimes v \otimes w_i^\ast \otimes X_{w_i}(a) \right] + - \otimes v \otimes e \otimes D(a).
\end{equation}
Suppose $v = \sum_{i=1}^m \alpha_i w_i + \beta_i w_i^\ast + \gamma e$.  The contraction projects tensor positions $\ell$ and $\ell+1$ and scales by the bilinear form applied to these positions. The only nonzero terms are of the form $\alpha_i \omega(v, w_i^\ast)$, $\beta_i \omega(v,w_i)$ and $\gamma \omega(v,e)$. Here $\omega$ is the bilinear form defined in Section \ref{Spinor Rep}. The sum in \eqref{TSEQ1} becomes
\[
\sqrt{2} \left[ \sum_{i=1}^m - \otimes - \otimes - \otimes \beta_i D_{w_i^\ast}(a) + - \otimes - \otimes - \otimes \alpha_i X_{w_i} \right] + - \otimes - \otimes - \otimes \gamma D(a).
\]
This is precisely the spin projection applied to the vector $v$ in position $k$. \qed
\end{proof}

\begin{lemma}\label{CliffEquivariant}
The spin-Clifford relation (Definition \ref{CliffordDef}) when applied to a diagram resulting from step \eqref{Mult7} of the multiplication process corresponds to interchanging the two corresponding spin injection(s) and/or projection(s).
\end{lemma}

\begin{proof}
Assume we have completed step \eqref{Mult7} of the multiplication process for some spin-Brauer diagrams. This produces a diagram, not necessarily spin-Brauer.  Lemmas \ref{CS2}-\ref{TSLemma} imply the simplifications in steps \eqref{Mult2}-\eqref{Mult7} agree with the corresponding simplification of maps.  Hence, these diagrams correspond to some map in the centralizer with the spin projections and injections out of order.  We will show that swapping the order of these maps corresponds to the spin-Clifford relation.

We first address swapping the order of isolated vertices across rows. If we want to swap two isolated vertices that occur in the same vertex position of the top and bottom row, i.e. a diagram of the form
{\tiny
\begin{equation} \label{SCEQ1}
\xymatrix@R=.06cm{
&\bullet \ar@{-}@/^/[r] &\bullet &\odot^2\;\\
\delta&&&\\
&\bullet \ar@{-}@/_/[r] &\bullet&  \odot^1.
}
\end{equation}}
This diagram must come from a composition where the isolated vertices follow through strings, for example a composition of the form
{\tiny
\begin{equation} \label{SCEQ2}
\xymatrix @R=.6pc{
\bullet \ar@{-}@/^/[r] & \bullet &\bullet\ar@{-}[ldd]\; \\
&&\\
\odot^1 &\bullet &  \odot^2\;\\
\odot^1 & \odot^2 &\bullet\; \ar@{-}[dd]\\
&&\\
\bullet \ar@{-}@/_/[r] &\bullet &  \bullet
.}
\end{equation}}
In this case, we notice that the entries in the tensor positions will remain the same. What we mean is that we will spin project the original entry in position three and then spin inject. What will change is the order of composition of the maps in the spin representation tensor position.  

If the isolated vertices are ordered consecutively, this means the corresponding linear operators will be composed consecutively. As a result, swapping the isolated vertex indexing corresponds to swapping the order of the composition of the linear operators in the spin representation tensor position.  Thus, the spin-Clifford relation reduces to the super commutativity of the operators.  

These linear operators, however, are not strictly super-commutative. $X_v$ and $D_\lambda$ satisfy the identity \eqref{DXCommute}. As a result, when we swap we must account for the other terms that appear in the sum. 

Suppose we want to apply a spin injection into position $\ell$ and projection from position $k$.  It suffices to restrict our attention to these tensor positions. So we consider $v \otimes - \otimes a$ where $v$ is in tensor position $k$ and we spin inject into tensor position $\ell$. We may assume this position has been projected away.   Indeed, if this is not the case it implies the diagram did not result from step \eqref{Mult7} in the multiplication of two diagrams. 

Let $v = \sum_{i=1}^m \alpha_i w_i + \beta_i w_i^\ast + \gamma e$. When we spin project $v$, it corresponds to applying the linear operators $\sqrt{2} \sum_{i=1}^m \alpha_i X_{w_i} + \beta_i D_{w_i^\ast} + \gamma D$ in the spin representation tensor position. If we spin inject into position $\ell$ then spin project position $k$ we have,
{\small \begin{align}
\begin{split}\label{SCEQ3}
 2 \sum_{j=1,i=1}^m &- \otimes \alpha_i w_j^\ast \otimes X_{w_i} X_{w_j}(a) + - \otimes \alpha_i w_j \otimes X_{w_i} D_{w_j^\ast}(a) + - \otimes \alpha_i e \otimes X_{w_i} D(a)\\
&+ - \otimes \beta_i w_j^\ast \otimes D_{w_i^\ast} X_{w_j}(a)+ - \otimes \beta_i w_j \otimes D_{w_i^\ast} D_{w_j^\ast}(a) +- \otimes  \beta_i e \otimes D_{w_i^\ast} D(a)\\
&+(\sqrt{2}) \left[- \otimes \gamma w_j^\ast \otimes D X_{w_j}(a) + - \otimes \gamma w_j \otimes D D_{w_j^\ast}(a) \right] + - \otimes \gamma e \otimes DD(a).
\end{split}
\end{align}}
Once again, we only consider the $k$ and $\ell$ tensor positions as well as the spin representation tensor position. When we swap all of these operators, we have the same sum but scaled by $(-1)$.  However, in any positions containing the linear operators $X_{w_i}$ and $D_{w_j^\ast}$, we get an extra term from \eqref{DXCommute} when $i=j$.  Explicitly, let $\Omega'$ be the sum in \eqref{SCEQ3}. When we swap all the operators we have
\[
-\Omega' + 2 \sum_{i=1}^m - \otimes \alpha_i w_i \otimes a + - \otimes \beta_i w_i^\ast \otimes a + - \otimes \gamma e \otimes a.
\]
Notice this simplifies to $-\Omega' +2( - \otimes v \otimes a)$ where $v$ is now in the $\ell$-tensor position and all of the other components of the diagram and corresponding maps are the same.  This shows that the following operations correspond:
\begin{itemize}
\item Swap the order of composition of a consecutively ordered spin projection and injection.
\item In the corresponding diagram, swap the indices of the isolated vertices and scale $(-1)$. Furthermore, add the diagram where we draw a through string between the isolated vertices we swapped.
\end{itemize}
This is the spin-Clifford relation.

If we wish to swap two spin injections in the bottom row the same proof applies.  Indeed, when we swap the indexing of consecutive isolated vertices, it corresponds to swapping the order of composition of the corresponding linear operators.  Once again, we scale by $(-1)$ and account for the additional terms that arise. However, in this case the additional terms result in the following sum,
\[
- \Omega ' + 2 \sum_{i=1}^m w_i \otimes w_i^\ast \otimes a + w_i^\ast \otimes w_i \otimes a + \be \otimes \be \otimes a.
\]
This second summand is exactly the spin immersion into the corresponding tensor positions. When $\dim(\bV)$ is even the same proof applies removing all of the terms that contain the vector $\be$. This establishes the spin-Clifford relation. \qed
\end{proof}

Lemmas \ref{CS2}-\ref{TSLemma} show that simplifying the inside homomorphism agrees with our multiplication structure. We then have a composition of the form (\ref{mapconvert}), but the projections and injections are not in the correct order. Lemma \ref{CliffEquivariant} proves that placing the spin projection and injection maps into the correct order agrees with the operation for swapping the indices in the diagrams.  As a result, each step of the simplification agrees.  We conclude that $f_{\Omega_2} \circ f_{\Omega_1} = f_{\Omega_2\Omega_1}$. \qed
\end{proof}

{\em Proof of Theorem \ref{TheoremA}}. Combine Theorem \ref{Iso} with Theorem \ref{THMMultAgrees}. \qed

\section{Cellularity of $\SB_n(\delta)$} \label{Cellularity}
For certain parameters $\delta= N$ and $n$ we just showed $\SB_n(N)$ surjects onto $\End_{\Pin(N)}(\bV^{\otimes n} \otimes \Delta)$ and is an isomorphism for $N \geq 2n$.  We will now prove $\SB_n(\delta)$ is a cellular algebra over any field $\bk$.  This will allow us to parametrize all of its irreducible representations.

Throughout this section we fix a field $\bk$.  Furthermore, $\bk\Sigma_\ell$ will denote the symmetric group algebra over $\bk$ on $\ell$ letters.

We refer the reader to Graham-Lehrer \cite{graham1996cellular} for the classical definition of a cellular algebra.  We will use K\"onig and Xi's basis-free characterization \cite{KX} to prove cellularity of $\SB_n(\delta)$.  Before we state this definition we need the following terminology.

\begin{definition}\cite[Definition 3.2]{xi1999partition}
Let $A$ be an $\bk$-algebra.  Assume there is an involution $i$ on $A$.  A two-sided ideal $J$ in $A$ is called a {\bf cell ideal} if and only if $i(J) = J$ and there exists a left ideal $\nabla \subset J$ such that $\nabla$ is finitely generated and free over $\bk$ and that there is an isomorphism of $A$-bimodules $\alpha: J \simeq \nabla \otimes_\bk i(\nabla)$ making the following diagram commute
\[
\xymatrix{
J \ar[rr]^\alpha \ar[d]_i &&\nabla \otimes_R i(\nabla) \ar[d]^{x \otimes y \mapsto i(y) \otimes i(x)}\\
J \ar[rr]^\alpha &&\nabla \otimes_R i(\nabla).
}
\]
The algebra $A$ (with involution $i$) is called {\bf cellular} if and only if there is an $k$-module decomposition $A = J_1' \oplus J_2' \oplus \cdots \oplus J_n'$ (for some $n$) with $i(J_j') =J_j'$ for each $j$ and such that setting $J_j = \bigoplus_{i=1}^j J_i'$ gives a chain of two-sided ideals of $A$: $0 = J_0 \subset J_1 \subset \cdots \subset J_n = A$ (each fixed by $i$) and for each $j$ the quotient $J_j' = J_j/J_{j-1}$ is a cell ideal with respect to the involution induced by $i$ on the quotient $A/J_{j-1}$.
\end{definition}

The $\nabla$ associated to each $J_j/J_{j-1}$ are called {\bf cell modules} or in \cite{graham1996cellular} {\bf cell representations}. With this terminology we recall an important lemma.

\begin{lemma}[Lemma 3.3, \cite{xi1999partition}] \label{NoBasisCellular}
Let $A$ be an algebra with an involution $i$.  Suppose there is a decomposition
\[
A = \bigoplus_{j=1}^m \bV_j \otimes_\bk \bV_j \otimes_k B_j
\]
where $\bV_j$ is a vector space and $B_j$ is a cellular algebra with respect to an involution $\alpha_j$ and a cell chain $J_1^{(j)} \subset \cdots \subset J_{s_j}^{(j)} = B_j$ for each $j$.  Define $J_t = \bigoplus_{j=1}^t \bV_j \otimes_k \bV_j \otimes_k B_j$.  Assume that the restriction of $i$ on $\bV_j \otimes_k \bV_j \otimes_k B_j$ is given by $w \otimes v \otimes b \mapsto v \otimes w \otimes \sigma_j(b)$.  If for each $j$ there is a bilinear form $\phi_j \colon V_j \otimes_k V_j \to B_j$ such that $\sigma_j(\phi_j(w,v)) = \phi_j(v,w)$ for all $w,v \in \bV_j$ and that the multiplication of two elements in $\bV_j \otimes \bV_j \otimes B_j$ is governed by $\phi_j \pmod{J_{j-1}}$, that is, for $x,y,u,v \in \bV_j$ and $b,v \in B_j$, we have
\[
(x \otimes y \otimes b)(u \otimes v \otimes v) = x \otimes v \otimes b \phi_j(y,u) c
\]
$\pmod {J_{j-1}}$ and if $\bV_j \otimes \bV_j \otimes J_\ell^{(j)} + J_{j-1}$ is an ideal in $A$ for all $\ell$ and $j$, then $A$ is a cellular algebra.
\end{lemma}

We will use this lemma to prove the following theorem.

\begin{theorem} \label{THMCellularity}
For any field $\bk$, $n \in \bZ_{\geq 0}$ and $\delta$ an arbitrary parameter $\SB_n(\delta)$ is a cellular algebra.
\end{theorem}

\begin{proof}
The proof will consist of many parts broken into lemmas and propositions showing $\SB_n(\delta)$ satisfies all the hypotheses of Lemma \ref{NoBasisCellular}.

We follow a general framework used in \cite{xi1999partition}.  We begin by introducing some notation.  For $n$ a positive integer, we denote by $E_n$:
\begin{align*}
E_n \; := \; &\bigg{\{}\rho = \left((\rho_1),(\rho_2),\dots,(\rho_k)\right) \; | \; \emptyset \not= (\rho_i) \subset \{1,\dots,n\},\bigcup_{i=1}^k (\rho_i) = \{1,\dots,n\},\\
& \; (\rho_i) \cap (\rho_j) = \emptyset \; (i \not= j), |(\rho_i)| \leq 2, k \in \bN \bigg{\}}.
\end{align*}
Consider the following example,
\begin{example}
\
\[
E_3 = \{(12)(3), (13)(2),(23)(1),(1)(2)(3)\}.
\]
Notice, we force each partition $\rho_i$ to contain at most two elements, so the partition $(123)$ is excluded. 
\end{example}

Define a counting function $m_1 \colon E_n \to \bN$ that counts the number of parts of a partition $\rho$ that have size $1$. So if $\rho = ((\rho_1)\dots(\rho_k))$, $m_1(\rho)$ is the number of $\rho_i$ such that $|\rho_i| = 1$. 

We now construct a vector space that will encode spin-Brauer diagrams.  Define a vector space $\bV_\ell$ with basis given by the set
\begin{equation} \label{VBasis}
\cS_\ell = \{(\rho,S) \; | \; \rho \in E_\ell, \; m_1(\rho) \geq \ell, \; S \subset \rho, \; |S| = m_1(S) = \ell\}.
\end{equation}
Basis elements are pairs $(\rho,S)$ with $\rho$ a partition of $n$ into pieces of size $1$ or $2$, such that there are at least $\ell$ partition elements of size one. Given such a partition $\rho$, we pair it with a subpartition $S$ consisting only of $(\rho_i)$ such that $|\rho_i| = 1$.

For $\Omega \in \cB(\SB_n(\delta))$, let $\ell$ be the number of through strings in $\Omega$.  We will associate to $\Omega$ a basis element $(x,S) \otimes (y,T) \otimes \sigma \in \bV_\ell \otimes \bV_\ell \otimes \bk\Sigma_\ell$ and show this association has an inverse and hence is a bijection.  

Label the vertices in the top and bottom rows of $\Omega$ as $\{1,\dots,n\}$.  The top and bottom rows of $\Omega$ partition $\{1,\dots,n\}$ into subsets of size $1$ or $2$ in a natural way. Every isolated vertex and originating vertex of a through string corresponds to a subset of size $1$.  Arcs correspond to subsets of size $2$. 

Let $x$ be the partition of the top row of $\Omega$ and $y$ the partition of the bottom row.  Next, put $S$ and $T$ as the subsets of vertex numbers in the top and bottom row respectively where through strings originate. Hence, $|S| = |T| = \ell$, i.e. $(x,S) \otimes (y,T) \in \bV_\ell \otimes \bV_\ell$.

Now $\sigma$ encodes how the through strings connect. Suppose 
\[
S = (S_1,S_2,\dots,S_\ell),  \qquad T = (T_1,T_2,\dots,T_\ell).
\]
Let $f$ be the bijection in the spin datum of $\Omega$. Put $\sigma$ as the permutation induced by $f$.  This gives a well defined element in $\bk\Sigma_\ell$. We then encode $\Omega$ as $(x,S) \otimes (y,T) \otimes \sigma \in V_\ell \otimes V_\ell \otimes \bk\Sigma_\ell$.

Conversely, if $(x,S) \otimes (y,T) \otimes \sigma \in \bV_\ell \otimes \bV_\ell \otimes \bk \Sigma_\ell$, we construct a spin-Brauer diagram $\Omega'$.  Let $U$ consist of all $(x_i) \in x\setminus S$ of size one.  Similarly, let $U'$ be all $(y_i) \in y \setminus T$ of size one.  Put $\Gamma$ as all $(x_i) \in x$ with $|x_i| = 2$.  Define $\Gamma'$ similarly.  The bijection $f$ is then induced by the permutation $\sigma$ on $S = x \setminus (U \cup V(\Gamma)) \to y \setminus (U' \cup V(\Gamma')) = T$.  Here $f$ sends the element $S_i \in S$ to $T_{(\sigma(i))} \in T$. This gives us a well defined spin datum and hence spin-Brauer diagram $\Omega' \in \cB(\SB_n(\delta))$. 

\begin{lemma} \label{encodingdiagrams}
This construction gives a bijective correspondence between $\Omega \in \cB(\SB_n(\delta))$ and the basis for $\bV_\ell \otimes \bV_\ell \otimes \bk\Sigma_\ell$ described in \eqref{VBasis}, where $\ell \in \{0,1,\dots,n\}$.
\end{lemma}

\begin{proof}
This construction is clearly invertible.
\end{proof}

\begin{remark}
As was pointed out by a reviewer, these constructions are not only invertible but also inverse to each other.
\end{remark}

Another piece of the cell-datum necessary to prove cellularity is an involution $i \colon \SB_n(\delta) \to \SB_n(\delta)$.  Given a spin-Brauer diagram $\Omega \in \cB(\SB_n(\delta))$ with spin-datum $(U,U',\Gamma,\Gamma',f)$ define the {\bf spin involution} $i$ as
\[
i(\Omega) = (U',U,\Gamma',\Gamma,f^{-1}).
\]
With the total order on $U'$ and $U$ reversed.  Extend $i$ by linearity to all of $\SB_n(\delta)$.  We leave it to the reader to verify that this is a well-defined element of $\SB_n(\delta)$.
\begin{lemma}
The linear map $i$ is an anti-automorphism of $\SB_n(\delta)$ with $i^2 = id$.
\end{lemma}

\begin{proof}
Linearity of the map is clear from the definition. Furthermore, by construction $i^2 = id$. It remains to check that $i(\Omega_1 \Omega_2) = i(\Omega_2) i(\Omega_1)$ on basis elements. Linearity will then imply the result in general.  This follows immediately from the realization of $\Omega_1$ and $\Omega_2$ as diagrams. The map $i$ exchanges the rows.  When realized as diagrams, it is clear the following operations are equivalent:
\begin{itemize}
\item Exchange the rows of both diagrams, swap which one is on top and take the product.
\item Take the product and exchange the rows.
\end{itemize}
To elaborate on this, clearly all through strings will remain the same. When we swap rows, and change which diagram is on top and then take the product we will end up with all the same closed circuits, everything will just be upside down. So after we apply all the Clifford relations, we will still get all of the same diagrams as in $\Omega_1 \Omega_2$ just with interchanged rows. So when we apply $i$ to $\Omega_1 \Omega_2$ we get the same diagrams.
\end{proof}

We now define a bilinear form $\phi_\ell: V_\ell \otimes V_\ell \to k\Sigma_\ell$.  Let $(\rho,S) \in \cS_\ell$ with $S = \{S_1,S_2,\cdots,S_\ell\}$ and $S_1 < S_2 < \cdots < S_\ell$.  Given $\mu \in E_n$ and $\nu \in E_m$ we define $\mu \cdot \nu$ to be the smallest partition created by merging all parts of $\mu$ and $\nu$ with common elements. For example,
\begin{example}
$\mu = \{(13),(2),(45)\}$ and $\nu = \{(12),(3),(4),(5)\}$, then $\mu \cdot \nu = \{(123),(45)\}$.  
\end{example}
Given a partition $\mu \in E_n$ let ${\rm sing}(\mu)$ be all the components of the partition with size $1$.  We call these the {\bf singletons} in the partition. 

Given $(x,S),(y,T) \in V_\ell$, consider 
\[
{\rm sing}(x) \setminus S = \{\gamma_1,\dots,\gamma_k\} \qquad \text{and} \qquad {\rm sing}(y) \setminus T = \{\gamma_{k+1},\dots,\gamma_{k+m}\}.
\]
Here we place $\gamma_i$ in the natural order corresponding to the elements they represent in the partition.  These are the lists of isolated vertices. Define $\Gamma_S$ to be the ordered set of $\gamma_j \in \{\gamma_1,\dots,\gamma_{k+m}\}$ such that there is some $1 \leq i,j \leq \ell$ where $S_i$ and $\gamma_j$ are in a component of $x \cdot y$. We define $\Gamma_T$ similarly.

Define $\beta(\Gamma_S,\Gamma_T)$ as the minimal number of pairs $(i+1,i)$ such that after inductively removing $\gamma_{i+1}$ from $\Gamma_S$ and $\gamma_{i}$ from $\Gamma_T$ and reindexing the $\gamma_j$, the remaining elements in $\Gamma_S$ are all less than the remaining elements of $\Gamma_T$. For example,

\begin{example}
If $\Gamma_S = \{\gamma_1,\gamma_2,\gamma_5,\gamma_6\}$ and $\Gamma_T = \{\gamma_3,\gamma_4,\gamma_7\}$, then $\beta(\Gamma_S,\Gamma_T) = 2$. Indeed, we first remove $(5,4)$ then after reindexing we have the sets $\{\gamma_1,\gamma_2,\gamma_4\}$ and $\{\gamma_3,\gamma_5\}$.  We remove $(4,3)$ and the order is correct.
\end{example}

The choice of pairs we remove is always uniquely determined. This gives the number of isolated vertices we need to swap across rows using the Spin-Clifford relation in our diagram multiplication. 

Finally, let ${\rm Cr}((x,S),(y,T))$ be the number of pairs $(i,j)$ so that $S_i$ and $T_j$ are contained in a component of $x \cdot y$. This component would have to be unique. This counts the number of through strings that connect in our two diagrams.

\begin{definition} \label{phi}
Define a map $\phi_\ell: \bV_\ell \otimes \bV_\ell \to \bk \Sigma_\ell$ by letting $\phi_\ell((x,S) \otimes (y,T))$ be zero if any of the following occur,
\begin{enumerate}[(1)]
\item  \label{Cond1}
There exists some $i,j$ with $1 \leq i,j \leq \ell$ and $i \not= j$ such that there is a part of $x \cdot y$ containing both $S_i$ and $S_j$.  Or dually, if there is a part of $x \cdot y$ containing both $T_i$ and $T_j$. 
\item \label{Cond2}
$|\Gamma_S| \not= |\Gamma_T|$. 
\item \label{Cond3}
${\rm Cr}((x,S),(y,T)) + |\Gamma_S| \not= \ell$ or equivalently ${\rm Cr}((x,S),(y,T)) + |\Gamma_T| \not= \ell$. 
\item \label{Cond4}
$\beta(\Gamma_S,\Gamma_T) \not= |\Gamma_S|$ or $\beta(\Gamma_S,\Gamma_T) \not= |\Gamma_T|$. Equivalently, ${\rm Cr}((x,S),(y,T)) + \beta(\Gamma_S,\Gamma_T) \not= \ell$.  
\end{enumerate}

Otherwise, let $\phi_\ell$ be the following element of $k\Sigma_{\ell}$. First, scale by $\delta$ for each component of $x \cdot y$ that does not contain any element from $S$ or $T$.  These elements correspond to closed circuits.  Next, scale by $2^{|\Gamma_S|} = 2^{|\Gamma_T|}$.  This accounts for the factor of two in the spin-Clifford relation.

Then, if there exists a component of $x\cdot y$ containing $S_i$ and $T_j$ let our permutation in $\Sigma_{\ell}$ send $i$ to $j$. Remove $S_i$ from $S$ and $T_j$ from $T$.

Since, we are assuming $\phi$ is nonzero, (\ref{Cond2}) and (\ref{Cond3}) imply the following
\begin{itemize}
\item Every remaining element of $S$ corresponds to some $\gamma_i \in \Gamma_S$.
\item Every remaining element of $T$ corresponds to some $\gamma_j \in \Gamma_T$. 
\end{itemize}
Indeed, if one of the remaining elements of $S$ does not correspond to an element of $\Gamma_S$ this means ${\rm Cr}((x,S),(y,T)) + |\Gamma_S| < \ell$. Similarly for $T$.

As in the definition of $\beta$, inductively remove pairs $(i+1,i)$ such that after removing $\gamma_{i+1}$ from $\Gamma_S$ and $\gamma_i$ from $\Gamma_T$ and reindexing, all the elements of $\Gamma_S$ are less than the elements of $\Gamma_T$.  During this process, suppose we remove the pair $(i+1,i)$. If $S_k$ corresponds to $\gamma_{i+1}$ and $T_m$ corresponds to $\gamma_{i}$, then our permutation sends $k$ to $m$. 

By (\ref{Cond4}) every element of $\Gamma_S$ will be less than the elements in $\Gamma_T$. Indeed, we need $\beta(\Gamma_S,\Gamma_T) = |\Gamma_S| = |\Gamma_T|$. Whence, this process gives a pairing of all the remaining elements of $S$ and $T$, i.e. a permutation in $k\Sigma_\ell$. Extend this map by linearity to $V_\ell \otimes V_\ell$.  
\end{definition}

As an example consider

\begin{example}
Let $\ell = 3$.  If $x = \{(1)(2)(3)(4)(5)(6)(78)\}$, $S = \{(1)(4)(6)\}$ and $y = \{(13)(2)(4)(5)(67)(8)\}$, $T = \{(2)(5)(8)\}$.  Then we have
\[
x \cdot y = \{(13)(2)(4)(5)(678)\}.
\]
We notice ${\rm Cr}((x,S),(y,T)) = 1$ because $S_3$ and $T_3$ are both in a component of $x \cdot y$.

Furthermore, ${\rm sing}(x) \setminus S = \{(2)(3)(5)\}$ and ${\rm sing}(y) \setminus T = \{(4)\}$.  With the notation we have been using, we say $\{(2)(3)(5)\} = \{\gamma_1,\gamma_2,\gamma_3\}$ and $\{(4)\} = \{\gamma_4\}$.  Now proceeding by definition we construct $\Gamma_S$ and $\Gamma_T$.  We see $(4) = S_2 = \gamma_4 \in x \cdot y$.  Also, $(1) = S_1$ and $\gamma_2$ are in the element $(13) \in x \cdot y$.  Hence $\Gamma_S = \{\gamma_2,\gamma_4\}$.

Next, we see $(2) = T_1= \gamma_1$ and $(5) = T_2 = \gamma_3$ are also both in $x \cdot y$.  Hence $\Gamma_T = \{\gamma_1,\gamma_3\}$.  This implies $\beta(\Gamma_S,\Gamma_T) = 2$ as we must remove $\gamma_2$ and $\gamma_1$ as well as $\gamma_3$ and $\gamma_4$ for $\Gamma_S$ to be less than $\Gamma_T$. It is important to remember we can only remove consecutively indexed elements.

We check that the bilinear form is nonzero.  The first condition does not occur.  $|\Gamma_S| = |\Gamma_T| = \beta(\Gamma_S,\Gamma_T) = 2$ and ${\rm Cr}((x,S),(y,T)) + \beta(\Gamma_S,\Gamma_T)  = 3 = \ell$.

There are no closed circuits because every element of $x \cdot y$ contains some $S_i$ or $T_j$. So we do not scale by a power of $\delta$. Now we must check what the permutation should be.  When computing the crossing number we saw that both $S_3$ and $T_3$ were in the component $(678)$ of $x \cdot y$.  Hence our permutation will fix $3$. 

To discover the final part of the permutation when finding $\beta(\Gamma_S,\Gamma_T)$ we had to remove both the pairs $(\gamma_2,\gamma_1), (\gamma_4, \gamma_3) \in \Gamma_S \times \Gamma_T$.  We saw that $S_1$ corresponds to $\gamma_2$ and $T_1$ corresponds to $\gamma_1$. Accordingly, our permutation fixes $1$.  Similarly, $S_2$ corresponds to $\gamma_4$ and $T_2$ corresponds to $\gamma_3$ so our permutation fixes $2$.  In conclusion, the resulting permutation is the identity.  So in this example $\phi((x,S),(y,T)) = 2^2 \cdot id \in k\Sigma_3$.
\end{example}

\begin{remark} \label{bilinearform}
The map $\phi_\ell: \bV_\ell \otimes_\bk \bV_\ell \to \bk\Sigma_\ell$ from Definition \ref{phi} is a bilinear form.
\end{remark}

We wish to show that multiplication of two diagrams with $\ell$ through strings is encoded by $\phi_\ell$ $\pmod{J_{\ell-1}}$ where $\phi_\ell$ is defined in Definition \ref{phi}.  Before we can prove this, we need the following result.

\begin{lemma}
$J_t := \sum_{j=0}^t \bV_j \otimes \bV_j \otimes k \Sigma_j$ is an ideal of $\SB_n(\delta)$.
\end{lemma}

\begin{proof}
This is stated in Koike \cite[p. 69]{koike2005spin}. Concretely, this is the ideal of all diagrams with at most $t$ through strings.  The number of through strings only decreases upon multiplication. \qed
\end{proof}

Let $\# \colon \SB_n(\delta) \to \bZ_{\geq 0}$ be the function that maps a sum of spin-Brauer diagram to the maximal number of through strings in the sum. We call $\#(\sum_i \Omega_i)$ the {\bf maximal crossing number}.

\begin{lemma} \label{ModMult}
Let $\Omega_1,\Omega_2 \in \cB(\SB_n(\delta))$.  If $\Omega_1 = (u,R) \otimes (x,S) \otimes \sigma_1 \in \bV_\ell \otimes \bV_\ell \otimes \bk\Sigma_\ell$ and $\Omega_2 = (y,T) \otimes (v,Q) \otimes \sigma_2 \in \bV_\ell \otimes \bV_\ell \otimes \bk\Sigma_\ell$, then
\[
\Omega_2 \Omega_1 = (u,R) \otimes (v,Q) \otimes \sigma_1 \phi_\ell((x,S),(y,T)) \sigma_2,
\]
modulo $J_{\ell-1} = \bigoplus_{j=0}^{\ell-1} \bV_j \otimes \bV_j \otimes \bk \Sigma_j$.
\end{lemma}

\begin{proof}
If $\psi_\ell((x,S),(y,T)) = 0$, then by definition of $\phi_\ell$ we see $\#(\Omega_2\Omega_1) < \ell$.  As in each situation \eqref{Cond1}-\eqref{Cond4} we lose a through string.  Furthermore, these are all the possible ways we could decrease the crossing number. This implies every element in the sum corresponding to $\Omega_2 \Omega_1$ is contained in $J_{\ell-1}$.  

Now assume $\phi_\ell((x,S),(y,T)) = 2^i \delta^j \sigma \in k\Sigma_\ell$ as defined in Definition \ref{phi}. It remains to show $2^i \delta^j (u,R) \otimes (v,Q) \otimes \sigma_1 \sigma \sigma_2$ corresponds to the element $\Omega_2 \Omega_1 \pmod{J_{\ell-1}}$.  

If $\#(\Omega_2\Omega_1) = \ell$ there will only be one diagram in the sum decomposition with $\ell$ through strings. Specifically, the diagram $\Omega$ in the sum resulting from repeatedly applying the spin-Clifford relation to swap isolated vertex indices across rows. All other diagrams will have less than $\ell$ through strings.  Indeed, after we apply the first spin-Clifford relation, the resulting diagram will have $|\Gamma_S| = |\Gamma_T| = \ell - C((x,S),(y,T)) - 1$. Hence, we can create at most $\ell-1$ through strings by applying the spin-Clifford relation.

As all other diagrams in the sum decomposition of $\Omega_2\Omega_1$ will have less than $\ell$ through strings it suffices to show 
\[
\Omega = 2^{|\Gamma_S|}\delta^k (u,R) \otimes (v,Q) \otimes \sigma_1 \sigma \sigma_2.
\]
The scalars are correct because we must apply the spin-Clifford relation exactly $|\Gamma_S|$ times by condition (\ref{Cond4}).  This results in scaling by $2^{|\Gamma_S|}$.  We also know $\Omega$ will be scaled by $\delta^k$ where $k$ is the number of closed circuits.  

Furthermore, $\Omega$ will have the same $(u,S)$ determining its top row and $(v,Q)$ its bottom row.  Indeed, $\Omega$ has $\ell$ through strings, so every through string has to be preserved.  Additionally, in the multiplication of $\Omega_2 \Omega_1$ we cannot change the originating vertex of a through string. This forces complete preservation of isolated vertices, arcs and through string origins. 

It remains to prove $\sigma_1 \sigma \sigma_2$ is the correct permutation of the through strings where we recall that we compose permutations left to right. First, if we consider a through string such that $S_i,T_j \in x \cdot y$ then by definition $\sigma(i) = j$.  Suppose $S_i$ is connected to $R_m$ and $T_j$ is connected to $Q_t$, i.e. $\sigma_1(m) = i$ and $\sigma_2(j) = t$.  In the final diagram, we need to send $R_m$ to $Q_t$. This clearly occurs as we compose from left to right $m \to i \to j \to t$.

Now consider the part of the diagram consisting of through strings that terminate in an isolated vertex. 
Suppose the through string originating in $R_m$ and connecting to $S_i$ ultimately terminates in an isolated vertex $\gamma_{k+1} \in \Gamma_S$.  Then by condition (\ref{Cond4}) after reindexing, we can assume without loss of generality that $\gamma_k \in \Gamma_T$.  This corresponds to some $T_j$ which is connected to $Q_t$. After applying the spin-Clifford relation, we create a through string between $R_m$ and $Q_t$.  Hence our permutation must send $m$ to $t$.  This is the case as $\sigma_1(m) = i$, then $\sigma(i) = j$ by definition and $\sigma_2(j) = t$.  We continue applying the spin-Clifford relation to generate all the other through strings in $\Omega$. By the same reasoning, every through string is correctly encoded by $\sigma_1\sigma\sigma_2$.

This proves $\Omega =  (u,R) \otimes (v,Q) \otimes \sigma_1 \phi_\ell((x,S),(y,T)) \sigma_2 \pmod{J_{\ell-1}}$. \qed
\end{proof}

\begin{lemma}
If $\Omega_1 = (x,S) \otimes (y,T) \otimes \sigma \in \bV_\ell \otimes \bV_\ell \otimes \bk\Sigma_\ell$ then $i(\Omega_1) = (y,T) \otimes (x,S) \otimes \sigma^{-1}$
\end{lemma}

\begin{proof}
This is a consequence of definitions and Lemma \ref{ModMult}. \qed
\end{proof}

\begin{lemma}
Let $\tau \colon \bk \Sigma_\ell \to \bk \Sigma_\ell$ be the involution on $\bk\Sigma_\ell$ defined by $\sigma \mapsto \sigma^{-1}$ for all $\sigma \in \Sigma_\ell$.  Then $\tau(\phi_\ell(v_1,v_2)) = \phi_\ell(v_2,v_1)$ for $v_i \in \bV_\ell$.
\end{lemma}

\begin{proof}
Assume $v_1 = (x,S)$ and $v_2 = (y,T)$. If $\phi_\ell(v_1,v_2) = 0$ then by construction $\phi_\ell(v_2,v_1) = 0$ as well. So assume $\phi_\ell(v_1,v_2) \not= 0$.

If this is the case, we notice that the scalar $\delta^k 2^{|\Gamma_S|}$ does not change when we interchange $S$ and $T$.  Indeed, the size of $\Gamma_S$ and $\Gamma_T$ are preserved.  Furthermore, the number of closed circuits remains the same because $x \cdot y = y \cdot x$. It remains to check that the permutation is inverted.

Let $\sigma$ be the permutation described in the construction of $\phi_\ell$.  If $S_i$ and $T_{\sigma(i)}$ are contained in the same part of $x \cdot y$ then $T_i$ and $S_{\sigma^{-1}(i)}$ are contained in the same part of $y \cdot x = x \cdot y$.  Hence the permutation associated to $\phi_\ell(v_2,v_1)$ is $\sigma^{-1}$ for through strings that connect.

Next we consider through strings that terminate in isolated vertices. Suppose $S_i$ and $T_{\sigma(i)}$ are associated to $\gamma_{k+1}$ and $\gamma_k$.  Without loss of generality, assume $\gamma_k$ is the largest element of $\Gamma_S$. When we apply $i$, the orders of $\Gamma_S$ and $\Gamma_T$ are reversed. That is,
\[
\gamma_{j} \to \gamma_{2k-j+1}.
\]
$S_i$ is now in a component of $x \cdot y$ with the isolated vertex corresponding to $\gamma_{2k-(k+1) + 1} = \gamma_k$.  Similarly, $T_{\sigma(i)}$ is associated to $\gamma_{k+1}$. We remove these elements and inductively apply the above reasoning to see that in general, $T_{\sigma(i)}$ and $S_{i}$ are mapped to each other. That is, the permutation associated to $\phi_\ell(v_2,v_1)$ is described by $\sigma(i) \mapsto i$. This permutation is $\sigma^{-1}$. \qed
\end{proof}

We are now ready to prove the theorem. Put $J_{-1} = 0$, $\Sigma_0 = \{1\}$ and $B_\ell = k \Sigma_\ell$.  Then $\SB_n(\delta)$ has a description
\[
\SB_n(\delta) = \bV_0 \otimes_k \bV_0 \otimes_k B_0 \oplus \cdots \oplus \bV_\ell \otimes_k \bV_\ell \otimes_k B_\ell \cdots \oplus \bV_n \otimes_k \bV_n \otimes_k B_n.
\]
This follows from Lemma \ref{encodingdiagrams}. 
Note that $B_r$ is a cellular algebra with respect to the involution $\sigma \mapsto \sigma^{-1}$ for $\sigma \in \Sigma_r$ (see \cite[Proof of Theorem, pg. 107]{xi1999partition}).  By all the Lemmas in this section this description of the spin-Brauer algebra satisfies all the necessary conditions in Lemma \ref{NoBasisCellular}. Hence $\SB_n(\delta)$ is a cellular algebra. \qed
\end{proof}

We extract some immediate consequences of cellularity. 

%
%

\begin{corollary} \label{WeylMods}
The cell modules of $\SB_n(\delta)$ are $\nabla_{\ell}(\lambda) := \bV_\ell \otimes v_\ell \otimes \nabla(\lambda)$ where $\ell \in \{0,1,\dots,n\}$ and $\lambda$ is a partition of $\ell$, $v_\ell$ is a fixed nonzero element of $\bV_\ell$ and $\nabla(\lambda)$ is a cell module of $\bk\Sigma_\ell$.  For $\ell = 0$, we take $\lambda = (0)$ and $\nabla(0) = k$.
\end{corollary}

\begin{proof}
This is an immediate consequence of the proof of Theorem \ref{THMCellularity}.  These are precisely the $\nabla$ associated to each $J_\ell/J_{\ell-1}$.  We also refer the reader to \cite{xi1999partition} for further discussion. \qed
\end{proof}

The existence of Weyl-Modules is an immediate consequence of the cellular structure. For a specific definition we refer the reader to \cite[Section 2]{graham1996cellular}.  These modules play a significant role in understanding the representations of a cellular algebra. In particular, we can define a symmetric bilinear form $\Phi_\ell$ on them.  This bilinear form is described using the bilinear form in the cellular datum \eqref{bilinearform}.  It turns out that over a field, non-degeneracy of this bilinear form for each $\ell$ is equivalent to semi-simplicity \cite[Theorem 3.8]{graham1996cellular}.  Furthermore, the $\ell$ for which $\Phi_\ell$ is non-degenerate parametrize all the absolutely irreducible representations of a cellular algebra \cite[Theorem 3.4]{graham1996cellular}.  

This means we can parametrize all the irreducible representations of $\SB_n(\delta)$.  First we must make a definition.

\begin{definition}
Let $a \in \bZ_{\geq 0}$.  A partition $\lambda= (\lambda_1,\dots,\lambda_\ell)$ of $m$ is {\bf $a$-regular} if there is no $\lambda_i = \lambda_{i+1} = \cdots = \lambda_{i+a-1}$ for any $i = 1,\dots, \ell-a$. This just means that every part of $\lambda$ appears at most $a-1$ times when $a>0$. Every partition is $0$-regular.
\end{definition}

\begin{corollary}
Let $\SB_n(\delta)$ be the spin-Brauer diagram algebra for some $n > 1$ over a field $k$ of characteristic $a$ where $a$ can be zero.  If $\delta\not= 0$, then the nonisomorphic irreducible representations of $\SB_n(\delta)$ are parametrized by $\{(m,\lambda) \; | \; 0 \leq m \leq n, \; \lambda \; \text{a $a$-regular partition of} \; m\}$. If $\delta= 0$ then $m = 0$.
\end{corollary}

\begin{proof}
We follow \cite{xi1999partition,graham1996cellular}.  From Corollary \ref{WeylMods} and \cite{graham1996cellular} the irreducible representations of $\SB_n(\delta)$ are parametrized by $\{(\ell,\lambda) \; | \; \Phi_{\ell,\lambda} \not= 0\}$. Here $\Phi_{\ell,\lambda}$ is a bilinear form on cell modules defined in \cite[\S 2]{graham1996cellular}. If $\ell \not= 0$ then $\Phi_{\ell,\lambda} \not= 0$ if and only if the corresponding linear form $\Phi_{\lambda}$ for the cellular algebra $k\Sigma_\ell$ is not zero.  Here we use the fact that $\phi_{\ell}((x,S),(x,S)) = \delta^{|x|-\ell} \; {\rm id} \in k\Sigma_\ell$. This implies from the definition of $\Phi_{\ell,\lambda}$ that for any $\ell \not= 0$ and $(x,S) \otimes v_\ell \otimes \nabla(\lambda)$ the bilinear form $\Phi_{\ell,\lambda}((x,S) \otimes v_\ell \otimes \nabla_1(\lambda),(x,S) \otimes v_\ell \otimes \nabla_2(\lambda)) = \delta^{|x|-\ell} \cdot \Phi_\lambda(\nabla_1(\lambda),\nabla_2(\lambda))$ which will be nonzero if and only if $\Phi_\lambda$ is nonzero because $\nabla_1(\lambda)$ and $\nabla_2(\lambda)$ are arbitrary standard modules of $k\Sigma_\ell$.

It follows from \cite[(7.6)]{DJ} that $\Phi_\lambda$ is nonzero if and only if $\lambda$ is a $a$-regular partition of $\ell$. If $m = 0$, then $\Phi_{\ell,\lambda} \not= 0$ if and only if $\delta \not= 0$. \qed
\end{proof}

\section{Further Questions}
Once $\SB_n(\delta)$ is defined separately from the centralizer algebra, it becomes possible to ask many questions about its structure--as seen above for the other diagram algebras.  In particular, we may inquire about 
\begin{itemize}[$\bullet$]
\item For which choices of $\delta$ and $n$ is $\SB_n(\delta)$ semisimple? We note that when $\delta = N$ and $N \geq 2n$ the semi-simplicity is clear because $\SB_n(N)$ is then isomorphic to $\End_{\Pin(N)}(\bV^{\otimes n} \otimes \Delta)$ which is semi-simple by the double centralizer theorem when $N \not= 2$. As with all the other diagram algebras, the interesting cases arise outside of these choices of $\delta$. 
\item When the map in Theorem \ref{TheoremA} is not an isomorphism, i.e. for $N < 2n$, what does the kernel of the map $\SB_n(N) \surjects \End_{\Pin(N)}(\bV^{\otimes n} \otimes \Delta)$ contain? 
\item Following \cite{BEG}, we can also ask which of the properties about $\SB_n(N)$ descend to $\End_{\Pin(N)}(\bV^{\otimes n} \otimes \Delta)$.
\end{itemize}

\end{document}